\pdfoutput=1

\documentclass[11pt]{amsart}

\usepackage[margin=1in]{geometry} 
\usepackage{amsmath,amsthm,amssymb,amsfonts,amscd}
\usepackage{tikz-cd}
\usepackage[latin2]{inputenc}
\usepackage{soul}
\usepackage{spectralsequences}
\usepackage{dsfont}

\usepackage{subcaption}
\usepackage{adjustbox}

\usepackage[bookmarks=false]{hyperref}

\usetikzlibrary{matrix}

\newcommand*\ZZ{|[draw,circle]| \mathbb{Z}/2}

\newcommand*\ZZZ{|[draw,circle]| \mathbb{Z}/8}
\newcommand*\ZZZZZ{|[draw,circle]| \mathbb{Z}/4}
\newcommand*\W{|[draw,circle]| \mathbb{Z}/3}

\tikzcdset{scale cd/.style={every label/.append style={scale=#1},
    cells={nodes={scale=#1}}}}

\newtheorem{theorem}{Theorem}[section]

\newtheorem{lemma}[theorem]{Lemma}
\newtheorem{proposition}[theorem]{Proposition}
\newtheorem{corollary}[theorem]{Corollary}

\newcommand{\bb}[1]{\mathbb{#1}}
\newcommand{\cal}[1]{\mathcal{#1}}

\newcommand{\ext}{\text{Ext}}

\newcommand{\Ker}{\text{Ker}}

\newcommand{\holim}{\text{holim}}
\newcommand{\hocolim}{\text{hocolim}}

\newcommand{\Top}{\text{Top}}
\newcommand{\Map}{\text{Map}}

\newcommand{\Sq}{\text{Sq}}
\newcommand{\Vect}{\text{Vect}}

\begin{document}

\title{Metastable Complex Vector Bundles over Complex Projective Spaces}

\author[Yang Hu]{Yang Hu}

\address{Department of Mathematics \\ University of Oregon \\ Eugene, OR 97403 \\ United States} 

\email{yhu2@uoregon.edu}

\date{\today}

\begin{abstract}

We apply Weiss calculus to count rank $r$ topological complex vector bundles over complex projective spaces
 $\bb{C}P^l$ with vanishing Chern classes when $r=l-1$ and $r=l-2$, which are the first two cases of the metastable range. 

\end{abstract}

\maketitle

\section{Introduction}

How many times must a boy look up, before he can see the sky? And how many rank $r$ topological complex vector bundles are there up to isomorphism, over a fixed finite dimensional CW complex? While the answer to the first question is forever blowin' in the wind, the answer to the latter can be approached using tools from homotopy theory.

Let  $\Vect_r(X) \cong [X, BU(r)]$ as usual denote the pointed set of rank $r$ bundles over a finite dimensional CW complex $X$, and let $\Vect_r^0(X)$ denote the subset of rank $r$ bundles whose Chern classes all vanish.

\begin{theorem}[The first unstable case] \label{firstunstable}
Let $l>2$ be an integer, and let $\psi(l)$ denote the cardinality of $\Vect_{l-1}^0(\bb{C} P^l)$.  Then
$\psi(l) = 2$ if $l$ is odd, and $\psi(l) = 1$ if $l$ is even.
\end{theorem}

\begin{theorem}[The second unstable case] \label{secondunstable}
Let $l>3$ be an integer, and let $\phi(l)$ denote the cardinality of $\Vect_{l-2}^0(\bb{C} P^l)$.  
The numbers $\phi(l)$ exhibit the following 24-fold periodic behavior.

\begin{center}
\begin{tabular}{c | c c c c c c c c c c c c}
\hline
$l \mod 24$ & $0$ & $1$ & $2$ & $3$ & $4$ & $5$ & $6$ & $7$ & $8$ & $9$ & $10$ & $11$ \\
\hline
$\psi(l)$ & $1$ & $1$ & $12$ & $2$ & $1$ & $3$ & $2$ & $2$ & $3$ & $1$ & $4$ & $6$ \\
\hline
\end{tabular}
\end{center}

\begin{center}
\begin{tabular}{c | c c c c c c c c c c c c}
\hline
$l \mod 24$ & $12$ & $13$ & $14$ & $15$ & $16$ & $17$ & $18$ & $19$ & $20$ & $21$ & $22$ & $23$ \\
\hline
$\phi(l)$ & $1$ & $1$ & $6$ & $2$ & $1$ & $3$ & $4$ & $2$ & $3$ & $1$ & $2$ & $6$ \\
\hline
\end{tabular}
\end{center}
\end{theorem}

Broadly speaking, there are two steps to proving these results.  
First,  in Theorem \ref{reduction}  we use
 Weiss calculus to identify stably trivial vector bundles over some $d$-dimensional complex 
 $X$ with $  \{ X, \Sigma \bb{C}P^{\infty}_r\}$, when 
 a bundle has rank $r$  with $\frac{d}{4}\leq r \leq \frac{d-1}{2}$, which we call the {\em metastable range}.  
 Here $\{ X, Y \}$ as usual denotes stable homotopy classes of maps, 
which is the direct limit $\varinjlim [\Sigma^n X, \Sigma^n Y]$, and $\bb{C}P^{\infty}_r$ is the stunted projective space $\bb{C}P^{\infty}/\bb{C}P^{r-1}$.

Weiss calculus is a framework which applies to some spaces given by evaluation of 
 functors on the category of vector spaces, resolving them by a tower of fibrations with infinite loop spaces as fibers. 
Resolution by infinite loop spaces, which are essentially abelian group objects, is a time-honored technique  in homotopy theory, with the Postnikov tower and unstable Adams resolutions being standard examples. The machinery of Weiss calculus gives a custom-made resolution of $BU(r)$ for any $r$, which in the metastable range translates to the stable mapping 
set above.

Stable mapping sets are amenable to standard tools such as Adams and Atiyah-Hirzebruch
spectral sequences, and the second part of our analysis is to employ such tools to make this calculation when $X$ is a complex project space, with needed incorporation of delicate calculations by Mosher \cite{Mosher}, Toda \cite{Toda}, and Matsunaga \cite{Matsunaga}.
Since stable mapping sets are abelian groups, the identification of Theorem~ \ref{reduction} 
equips $\Vect_r^0(X)$ with an abelian group structure in the metastable range.
We  calculate  these groups, which are cyclic of the orders given. We  prefer to present our main results in terms of cardinality, since we don't have an intrinsic description of a group structure on $\Vect_r^0(X)$.  Other open questions invited by our work include finding  representatives for these isomorphism classes of bundles, in particular finding
whether these have holomorphic representatives, as well as finding invariants distinguishing these bundles.  For some of these questions, it
might be helpful to have
a more direct comparison between $\Vect_r^0(X)$ and stable maps to stunted projective spaces.

For perspective on our techniques, recall that rank $r$ bundles over projective spaces with trivial Chern classes are measured by $[\bb{C}P^l, U/U(r)]$, since  vanishing of Chern classes in this case
 implies stable
triviality, and the homogeneous space $U/U(r)$ is the homotopy fiber of the standard map $BU(r)\rightarrow BU$.
One can thus imagine directly calculating $[\bb{C}P^l, U/U(l-1)]$  and $[\bb{C}P^l, U/U(l-2)]$. 
However, these calculations are unstable in nature and for example the unstable Adams spectral sequences computing them are not accessible. 
Weiss calculus enables us to replace the target spaces $U/U(r)$ by infinite loop spaces $Q\Sigma\bb{C}P^{\infty}_r$ (where $Q(-) = \Omega^{\infty}\Sigma^{\infty}(-)$), and hence to extract the desired unstable information from stable calculations.

\subsection{Review of Previous Work}

We first fix notation. 
The cohomology ring of $\bb{C}P^l$ is isomorphic to the truncated polynomial algebra $\bb{Z}[x]/(x^{l+1})$, where $x$ in degree two is identified with the first Chern class of the 
dual of the tautological line bundle. 
The Chern classes of any bundle $\xi$ over $\bb{C}P^l$ are integer multiples of the powers of $x$, and will therefore be treated sometimes as integers.

Classification of topological complex vector bundles is typically organized around $K$-theory and Chern classes, which in special cases give complete information.
When $2r\geq \dim X$,  bundles are stable -- that is, isomorphism is equivalent to stable isomorphism -- and can hence be studied through $K$-theory. 
When $r=1$, line bundles are determined by the first Chern class $c_1$. 
However, when $1<r<\frac{1}{2}\dim X$, rank $r$ bundles over $X$ are  much harder to compute and detect.
For example, the calculation
$$\Vect_2(S^6) \cong [S^6, BU(2)] \cong [S^5, SU(2)] \cong \pi_5(S^3) \cong \bb{Z}/2$$
implies that there is a nontrivial rank 2 bundle over the $6$-sphere, but such a bundle can only have vanishing Chern classes.

Since the classifying spaces $BU(r)$ are rationally formal, there will only be finitely many rank $r$ bundles with a given set of Chern classes. 
We call the counting of this set \emph{Chern enumeration}. One key special case of Chern enumeration is
 finding the cardinality of $\Vect_r^0(X)$, which we call  \emph{vanishing Chern enumeration} question.  Our main results,   Theorems \ref{firstunstable} and  \ref{secondunstable},  resolve the  vanishing Chern enumeration question for complex projective spaces,  in some of the cases in which the Weiss tower for $BU(-)$ collapses to a single fibration -- see Theorem~\ref{reduction}. 

Another key case of Chern enumeration is determining when the number of vector bundles
with fixed Chern classes is non-zero, which we call the \emph{Chern realization question}.
Such results are given by arithmetic conditions on the Chern classes.
For example, Alan Thomas \cite{Thomas} proves that integer pairs $(c_1, c_2)$ can serve as the Chern classes for some stable bundle over $\bb{C}P^l$ with $c_i = 0$ for $i\geq 3$ precisely when the Schwarzenberger condition \cite{Hirzebruch} is satisfied.

Starting with Chern realization results and then applying $K$-theory techniques, Atiyah and Rees \cite{Atiyah-Rees} completely resolve the Chern enumeration question for rank
two bundles over $\bb{C}P^l$ for $l=3$ and $l=4$.

\begin{theorem}[Atiyah-Rees '76 \cite{Atiyah-Rees}]
Let $\xi \downarrow \bb{C}P^3$ be a stable bundle with $c_i(\xi) = 0$, $i\geq 3$.
\begin{enumerate}
	\item If $c_1(\xi)$ is even, then $\xi$ has exactly two rank 2 representatives.
	\item If $c_1(\xi)$ is odd, then $\xi$ has a unique rank 2 representative.
\end{enumerate}
\end{theorem}

\begin{theorem}[Atiyah-Rees '76 \cite{Atiyah-Rees}]
Every stable bundle $\xi \downarrow \bb{C}P^4$ with $c_i(\xi) = 0$, $i\geq 3$, contains a unique rank 2 representative.
\end{theorem}

As special cases of these theorems, one obtains vanishing Chern enumeration results. We will give alternate
proofs of these in Section 3, as illustrative first cases of our more general results.

\begin{corollary} \label{cor3}
Over $\bb{C}P^3$ there is a unique nontrivial rank 2 bundle with vanishing Chern classes. 
\end{corollary}

\begin{corollary} \label{cor4}
Over $\bb{C}P^4$ there is no nontrivial rank 2 bundle with vanishing Chern classes. 
\end{corollary}

It is later proved, both by Rees \cite{Rees} and by Smith \cite{Smith}, that for every $l\geq 5$ there exists some nontrivial rank 2 bundle over $\bb{C}P^l$ with vainishing Chern classes, so that vanishing Chern enumeration is nontrivial.
Using obstruction-theoretic techniques, Switzer \cite{Switzer} gives alternative proofs of the Atiyah-Rees theorems and goes further to resolve Chern enumeration for rank two bundles over $\bb{C}P^5$ and $\bb{C}P^6$. 

Chern enumeration over projective spaces remains a mystery in general, but is advancing not only in the 
present work but also in \cite{Opie}. In this work, Opie uses the theory of topological modular forms
to not only fully resolve Chern enumeration for rank 3 bundles over $\bb{C}P^5$ but also find invariants that distinguish bundles having the same Chern classes.  Moreover, Opie concretely constructs interesting rank 3 bundles over $\bb{C}P^5$.

\subsection{Plan of the Paper} 
In Section 2 we give a brief account of Weiss calculus.
The classifying space $BU(r)$ is the  value of the functor $BU(-): V\mapsto BU(V)$ at the standard vector space $\bb{C}^r$. 
To detect rank $r$ bundles over $\bb{C}P^l$ we map $\bb{C}P^l$ to the Weiss tower of $BU(-)$ evaluated at $V = \bb{C}^r$.
The Weiss tower of $BU(-)$ is studied in detail by Arone \cite{Arone02}. We shall make use of the description of the first layer and the cohomology calculation of the general layers to prove our identification $\Vect^0_r(\bb{C}P^l)\cong \{\bb{C}P^l, \Sigma \bb{C}P^{\infty}_r\}$ in the metastable range. 

For those less familiar with stable homotopy techniques, in Section 3 we give illustrative examples of our main results,
presenting new proofs of Corollaries \ref{cor3} and \ref{cor4}, which enumerate  $\Vect_2^0(\bb{C}P^3)$ and $\Vect_2^0(\bb{C}P^4)$.
After our stable map identification theorem, these results are proved by showing that 
$\{\bb{C}P^3, \Sigma\bb{C}P^{\infty}_2\} \cong \bb{Z}/2$ and $\{\bb{C}P^4, \Sigma\bb{C}P^{\infty}_2\} = 0$.
To do so we apply the 2-primary Adams spectral sequence to compute some first  stable homotopy groups of $\Sigma\bb{C}P^{\infty}_2$, and consider these as coefficients of the cohomology theory represented by the infinite loop space $Q\Sigma\bb{C}P^{\infty}_2$. 
The computation is then finished with an Atiyah-Hirzerbruch spectral sequence argument.

Section 4 is dedicated to the proof of our main results, Theorems \ref{firstunstable} and \ref{secondunstable}, following the strategy of Section 3. 
Rank $(l-1)$ and $(l-2)$ bundles over $\bb{C}P^l$ are in the metastable range, and the associated Weiss tower 
yields identifications 
$$\Vect_{l-1}^0(\bb{C}P^l) \cong \{\bb{C}P^l, \Sigma\bb{C}P^{\infty}_{l-1}\} \quad \text{and} \quad \Vect_{l-2}^0(\bb{C}P^l) \cong \{\bb{C}P^l, \Sigma\bb{C}P^{\infty}_{l-2}\}. $$
In the first  case, Theorem \ref{firstunstable}, we apply the 2-primary Adams spectral sequence to compute some first few stable homotopy groups of $\Sigma\bb{C}P^{\infty}_{l-1}$ (with some details presented in the Appendix), and then organize the computations of $\{\bb{C}P^l, \Sigma\bb{C}P^{\infty}_{l-1}\}$ with the Atiyah-Hirzebruch spectral sequence. 
The proof of Theorem \ref{secondunstable} has similar ingredients, but with some added complexity for two reasons. First, both the prime 2 and the prime 3 are involved. Secondly, more detailed study of 3-cell stunted projective spaces is required in order to determine a crucial $d_4$-differential in the Atiyah-Hirzebruch spectral sequence, where we make use of results of Mosher \cite{Mosher}.

\section*{Acknowledgments}

The author would like to thank his advisor Dev Sinha for his enormous amount of support in this research work, 
Morgan Opie for sharing her thesis which helps the author discover an error in an earlier version of this paper, 
Robert Bruner for running his computer code to help us double check our hand-calculated Ext-charts, and 
Tom Goodwillie and Michael Weiss for helpful correspondence.

\section{Weiss Calculus and the Metastable Exact Sequences}

We apply Weiss calculus to equate the vanishing Chern enumeration problem over complex projective spaces with the calculation
of a stable homotopy mapping set, in the metastable range. 
The following simplification of the Weiss tower in the case at hand is key to our metastable computations.

\begin{theorem}\label{reduction}
Let $l> 2$ be an integer, and let $r$ be an integer with $\frac{l}{2}\leq r \leq l-1$. 
Then the Weiss tower of the functor $V\longmapsto \Map_*\big{(}\bb{C}P^l, BU(V)\big{)}$
collapses to a single fibration after evaluating at $V=\bb{C}^r$, and there is the following exact sequence
$$0\longrightarrow \{ \bb{C}P^l, \Sigma \bb{C}P^{\infty}_r \} \longrightarrow [\bb{C}P^l, BU(r)] \longrightarrow [\bb{C}P^l, BU],$$
which we call the metastable exact sequence. It then follows that
$$\Vect_r^0(\bb{C}P^l) \cong \{ \bb{C}P^l, \Sigma \bb{C}P^{\infty}_r \}.$$
\end{theorem}

To prove Theorem~\ref{reduction}, we begin with a brief account of Weiss calculus.

\subsection{Fundamentals of Weiss Calculus}
Weiss initiated the study of his calculus, inspired by Goodwillie calculus, in  \cite{Weiss}.  There he
focuses on orthogonal calculus, but we apply unitary calculus here.
Let $\cal{J}$ be the category whose objects are finite dimensional complex vector spaces with positive definite inner product (all in a fixed universe $\bb{C}^{\infty}$), and whose morphisms are linear isometric inclusions. We consider $\cal{J}$ as a topological category since its morphism sets are Stiefel manifolds. 
Weiss calculus studies continuous functors from $\cal{J}$ to pointed spaces. 

For each $n\geq 0$ there is a distinguished class of {\em $n$-polynomial} functors, and the idea of calculus is to approximate a general functor by these polynomial ones -- analogous to the philosophy of classical calculus. Each continuous functor $F: \cal{J} \rightarrow \Top_*$ is equipped with a tower of fibrations
$$\cdots \longrightarrow T_nF \longrightarrow T_{n-1}F \longrightarrow \cdots \longrightarrow T_1F \longrightarrow T_0F$$
where $T_nF$ is $n$-polynomial, called the Weiss tower of $F$. 
For each $n$ there is a comparison map $F\rightarrow T_nF$ compatible with the tower, regarded as the universal approximation of $F$ by an $n$-polynomial functor.
We say the tower converges to $F$ if $F \rightarrow \holim_n T_nF$ is a weak equivalence.
With appropriate connectivity conditions, some towers terminate after finitely many steps.

The homotopy fiber $L_nF:= \text{hofib } (T_nF \rightarrow T_{n-1}F)$ is an {\em $n$-homogeneous}
 functor, which we call the {\em $n$-th Weiss layer} of $F$.
A fundamental theorem of Weiss calculus is that $n$-homogeneous functors are classified by $U(n)$-spectra, and in particular $L_nF$ is of the form
$$V \longmapsto \Omega^{\infty}(\Theta \wedge S^{nV})_{hU(n)}$$
where $\Theta$ is some $U(n)$-spectrum, $nV=\bb{C}^n\otimes_{\bb{C}}V$ with $U(n)$ acting on the left, and $S^{nV}$ denotes the one point compactification.
The classifying spectrum $\Theta$ is called the {\em $n$-th Weiss derivative} of $F$.

The Weiss tower, like the Goodwillie tower, can be viewed as a tool which resolves unstable structures
by stable ones.
While the values of these functors can be viewed as  unstable homotopy types,  the layers are infinite loop spaces by
the above classification theorem, and thus stable. Moreover, the bottom $T_0F$ of the tower is by definition $(T_0F)(V) = \hocolim_k F(V\oplus \bb{C}^k)$, which is manifestly a stabilization.
For $F(V) := BU(V)$ the bottom layer
 $T_0F$ is the constant functor $V\mapsto BU$, the classifying space for stable bundles.

\subsection{Identification of Derivatives}
We build on the seminal work of Arone \cite{Arone02} on the derivatives of the functor $F(V):= BU(V)$, whose $n$-th derivative is denoted by $\Theta_n$, $n\geq 1$. We denote by $\bb{L}_n$ the unreduced suspension of the realization of the category of non-trivial direct-sum decompositions of $\bb{C}^n$, and by $\text{Ad}_n$ the adjoint representation of $U(n)$. 
The following result provides a closed-form description of $\Theta_n$.

\begin{theorem}[Arone '02 \cite{Arone02}] \label{Weiss-derivative}
For every $n\geq 1$,  $\Theta_n$ is equivalent to $\Map_*(\bb{L}_n, \Sigma^{\infty}S^{\text{Ad}_n})$.
\end{theorem}

It follows immediately that the $n$-th layer of $F$ is of the form
$$(L_nF)(V) \simeq \Omega^{\infty}\Map_*(\bb{L}_n, \Sigma^{\infty}S^{\text{Ad}_n}\wedge S^{nV})_{hU(n)}.$$

For a general $n$, the $n$-th derivative $\Map_*(\bb{L}_n, \Sigma^{\infty}S^{\text{Ad}_n})$ need not have a homotopy type we can readily describe.
However, the first derivative, and hence the first layer, can be made explicit in a way which
 is of fundamental importance to this paper.

\begin{proposition} \label{1stlayer}
The first Weiss layer $(L_1F)(V)$ of $F$ is equivalent to $Q(\Sigma S^V)_{hU(1)}$, which in turn is equivalent to $Q\Sigma \bb{C}P^{\infty}_r$.
\end{proposition}

\begin{proof}
By Theorem \ref{Weiss-derivative}, $(L_1F)(V)$ is equivalent to  
$\Omega^{\infty}\Map_*(\bb{L}_1, \Sigma^{\infty}S^{\text{Ad}_1}\wedge S^{V})_{hU(1)}$.
Since the space $\bb{L}_1$ is just $S^0$, and $S^{\text{Ad}_1}$ is $S^1$ with trivial $U(1)$-action, 
$(L_1F)(V)$ is identified with $Q(\Sigma S^V)_{hU(1)}$.
To establish the second equivalence, we first observe that
when $V=\bb{C}^r$ the space $S^V_{hU(1)} = S^{2r}_{hU(1)}$ is the Thom space of the vector bundle 
$\gamma^{\oplus r}$ over $BU(1) = \bb{C}P^{\infty}$, where $\gamma$ is the tautological line bundle. 
Indeed, the action of $U(1)$ on $S^{2r} = S^V|_{V=\bb{C}^r}$ restricts to 
scalar multiplication of $U(1)$ on $\bb{C}^r$. Therefore the unreduced homotopy orbit
$$uS^{2r}_{hU(1)} := EU(1) \times_{U(1)} S^{2r} = \big{(}EU(1) \times_{U(1)} S^V \big{)}|_{V=\bb{C}^r}$$
is the fiberwise one-point compactification of the vector bundle $\gamma^{\oplus r}$ over $\bb{C}P^{\infty}$. The reduced homotopy orbit $S^{2r}_{hU(1)}$ is obtained from $uS^{2r}_{hU(1)}$ by collapsing the section of $\bb{C}P^{\infty}$ at infinity, and is therefore the desired Thom space.
Secondly, we recall (say from Proposition 4.3 of \cite{Atiyah}) that $\bb{C}P^{r+N}/\bb{C}P^{r-1}$ is the Thom space of $\gamma_{1, N}^{\oplus r}$, where $\gamma_{1, N}$ denotes the canonical line bundle over $\bb{C}P^{N}$. Letting $N\rightarrow \infty$ identifies the Thom space of $\gamma^{\oplus r}$ over $\bb{C}P^{\infty}$ with $\bb{C}P^{\infty}_r = \bb{C}P^{\infty}/\bb{C}P^{r-1}$. We conclude that $S^{2r}_{hU(1)} \simeq \bb{C}P^{\infty}_r$, and hence that
$$(L_1F)(\bb{C}^r) \simeq Q\Sigma S^{2r}_{hU(1)} \simeq Q\Sigma \bb{C}P^{\infty}_r.$$
\end{proof}

To study bundles over $\bb{C}P^l$, $l>2$, we consider the Weiss tower for the following type of functors: 
$$F^l: V\longmapsto \Map_* \big{(} \bb{C}P^l, BU(V) \big{)}, \quad l>2,$$
as $[\bb{C}P^l, BU(r)]$ is just $\pi_0F^l(\bb{C}^r)$.
The Weiss tower of $F^l$ can be obtained by mapping $\bb{C}P^l$ to the Weiss tower of $F(-)=BU(-)$.

\begin{proposition} \label{mappingtower}
The following are equivalences:
$$(T_nF^l)(V) \simeq \Map_*\big{(}\bb{C}P^l, (T_nF)(V)\big{)}, \text{and } (L_nF^l)(V) \simeq \Map_*\big{(}\bb{C}P^l, (L_nF)(V)\big{)}.$$
\end{proposition}
\begin{proof}
By the construction of Weiss calculus (see Section 5 of \cite{Weiss}), $T_nF$ is a direct homotopy colimit of some homotopy limits. 
For every finite complex $X$ the second-variable mapping functor $\Map_*(X, -)$ preserves arbitrary homotopy limits and filtered homotopy colimits.
\end{proof}

Thus the Weiss tower of $F^l$ can be presented by the following diagram.
\begin{center}
\begin{tikzcd}[row sep = small, column sep = small]
 & \vdots \ar[d] & \\
 & \Map_*\big{(}\bb{C}P^l, (T_2F)(V)\big{)} \ar[d] & \Map_*\big{(}\bb{C}P^l, (L_2F)(V)\big{)} \ar[l]\\
 & \Map_*\big{(}\bb{C}P^l, (T_1F)(V)\big{)} \ar[d] & \Map_*\big{(}\bb{C}P^l, (L_1F)(V)\big{)} \ar[l] \\
\Map_*\big{(}\bb{C}P^l, BU(V)\big{)} \ar{r} \ar[ru, bend left = 5] \ar[ruu, bend left = 15] & \Map_*(\bb{C}P^l, BU) & 
\end{tikzcd}
\end{center}

Convergence still holds because the original tower becomes more highly connected in each layers (see Section 2.3). The base $\Map_*(\bb{C}P^l, BU)$ of the Weiss tower is generally not a connected space. However, in this paper we  work over the base point component, namely the component of the stable trivial bundle.

\subsection{Cohomology Descriptions of Layers}

An insight of Arone \cite{Arone02} is that the spectra $\Theta_n \simeq \Map_*(\bb{L}_n, \Sigma^{\infty}S^{\text{Ad}_n})$, appearing here as Weiss derivatives of the functor $F(-) = BU(-)$, are closely related to the Goodwillie derivatives of the identity functor. 
Building on prior work including that of Arone \cite{Arone98}, Arone-Dwyer \cite{Arone-Dwyer} and Arone-Mahoword \cite{A-M}, 
the following cohomology description of $\Theta_n$, and hence that of the layers $L_nF$, is established in \cite{Arone02}.
In what follows, $\cal{A}_{k-1}$ denotes the subalgebra of the mod $p$ Steenrod algebra $\cal{A}$, generated by elements $\Sq^1, \Sq^2, \cdots, \Sq^{2^{k-1}}$ if $p=2$ and by elements $\beta, \cal{P}^1, \cal{P}^p, \cdots, \cal{P}^{p^{k-2}}$ if $p$ is odd. 

\begin{theorem}[Arone '02 \cite{Arone02}] \label{Arone}
The spectrum $\Theta_n$ is rationally contractible for $n>1$. Integrally, it is contractible unless $n$ is a prime power. If $n = p^k >1$ then the homology of $\Theta_{p^k}$ is all $p$-torsion, and the mod $p$ cohomology of $\Theta_{p^k}$ is free over $\cal{A}_{k-1}$.

Over $\cal{A}_{k-1}$ the cohomology of the homotopy orbit spectrum $(\Theta_{p^k}\wedge S^{p^kV})_{hU(p^k)}$ is free, and is isomorphic (up to appropriate suspension depending on $V$) to $\cal{A}_{k-1}\otimes \bar{P}$, where $\bar{P}$ is a free module on one generator over the polynomial algebra $\bb{F}_p[d_{-\infty}, d_0, \cdots, d_{k-1}]$, with $|d_j| = 2(p^k-p^j)$.
\end{theorem}

The rational contractibility alternatively follows from the fact that the spaces $BU(n)$ are rationally formal, so Chern
classes, which are pulled back from the bottom layer of the tower $BU$, determine vector bundles.
The main conclusion we shall draw from Theorem \ref{Arone} is the connectivity of the layers. 
If $n$ is not a prime power then $(L_nF)(V)$ is contractible. 
When $n$ equals the prime power $p^k$, the lowest nontrivial reduced cohomology of $(L_{p^k}F)(V)$ appears at degree $1 + 2p^k\cdot \dim_{\bb{C}}(V)$. 
The following connectivity estimate holds.

\begin{corollary} \label{ctd}
Let $r$ be the dimension of $V$. Then $(L_2F)(V)$ is $4r$-connected. The higher layers $(L_iF)(V), i\geq 3$ are more than $4r$-connected.
\end{corollary}

\subsection{Proof of Theorem~\ref{reduction}} 

Let  $l> 2$, and  $r$ be such that $\frac{l}{2}\leq r \leq l-1$.  Consider the Weiss tower of 
$F^l(V) = \Map_*\big{(}\bb{C}P^l, BU(V)\big{)}$ at $V = \bb{C}^r$, which by
Proposition~\ref{mappingtower} is obtained by mapping $\bb{C}P^l$ to the Weiss tower of $F(V)=BU(V)$.
According to the connectivity estimate of Corollary \ref{ctd}, the layers $L_nF(\bb{C}^r)$ of the tower are at least $4r$-connected for all $n\geq 2$.
Since $\bb{C}P^l$ is $2l$-dimensional and $2l \leq 4r$, spaces
$$(L_nF^l)(\bb{C}^r) \simeq \Map_*\big{(}\bb{C}P^l, L_nF(\bb{C}^r)\big{)}$$
are connected for $n\geq 2$ and hence the connected components of the Weiss tower of $F^l$ at $V=\bb{C}^r$ stabilize after the first stage.
Recall from Proposition \ref{1stlayer} that $(L_1F)(V) \simeq Q(\Sigma S^V)_{hU(1)}$, which is $QS^{2r+1}_{hU(1)} \simeq Q\Sigma \bb{C}P^{\infty}_r$ when evaluated at $V=\bb{C}^r$. So we conclude that the first layer $L_1F^l$ of $F^l$ at $V=\bb{C}^r$ is of the form
$$(L_1F^l)(\bb{C}^r) \simeq \Map_*\big{(}\bb{C}P^l, QS^{2r+1}_{hU(1)}\big{)} \simeq \Map_*(\bb{C}P^l, Q\Sigma \bb{C}P^{\infty}_r).$$

To sum up, the first stage of the Weiss tower of $F^l$ at $V=\bb{C}^r$ consists of the   fibration
$$\Map_*\big{(}\bb{C}P^l, QS^{2r+1}_{hU(1)}\big{)} \longrightarrow \Map_*\big{(}\bb{C}P^l, BU(r)\big{)} \longrightarrow \Map_*\big{(}\bb{C}P^l, BU\big{)}.$$
Taking components and using the connectedness of higher layers yields the following exact sequence
$$0\longrightarrow [\bb{C}P^l, QS^{2r+1}_{hU(1)}] \longrightarrow [\bb{C}P^l, BU(r)] \longrightarrow [\bb{C}P^l, BU].$$

This exactness implies that the subset of rank $(l-1)$ bundles over $\bb{C}P^l$ which stabilize to the trivial bundle, whihc
for complex projective spaces are exactly those with trivial Chern class data, is identified with the abelian group $[\bb{C}P^l, QS^{2r+1}_{hU(1)}]$. Since this is just $\{ \bb{C}P^l, \Sigma \bb{C}P^{\infty}_r \}$, Theorem \ref{reduction} is proved.

\section{First Cases: Rank Two Bundles over \texorpdfstring{$\bb{C}P^3$}{CP3} and \texorpdfstring{$\bb{C}P^4$}{CP4}}

The goal of this section is to make calculations with the Weiss tower to recover the two classical examples, Corollaries \ref{cor3} and \ref{cor4}. 
Namely, the vanishing Chern enumeration of rank 2 bundles over $\bb{C}P^3$, and that of rank 2 bundles over $\bb{C}P^4$.
By Theorem~\ref{reduction} we have the identifications
$$\Vect_2^0(\bb{C}P^3) \cong \{\bb{C}P^3, \Sigma\bb{C}P^{\infty}_2\} \quad \text{and} \quad \Vect_2^0(\bb{C}P^4) \cong \{\bb{C}P^4, \Sigma\bb{C}P^{\infty}_2\}.$$

We prove Corollaries \ref{cor3} and \ref{cor4} by showing that
$$\{\bb{C}P^3, \Sigma\bb{C}P^{\infty}_2\} \cong \bb{Z}/2 \quad \text{and} \quad \{\bb{C}P^4, \Sigma\bb{C}P^{\infty}_2\} = 0.$$ 

We shall regard these as generalized cohomology computations and apply the Atiyah-Hirzebruch spectral sequence.
To learn the coefficient ring, namely the stable homotopy groups of $\Sigma\bb{C}P^{\infty}_2$, we apply the (2-primary) Adams spectral sequence.
For those unfamiliar with the Adams spectral sequence, the recent expository paper of Beaudry-Campbell \cite{Beaudry-Campbell} provides a wonderful introduction.
We start by describing the action of the Steenrod squares on the cohomology of $\Sigma\bb{C}P^{\infty}_2$, and then construct an explicit minimal resolution to compute the Adams $E_2$-page through a range. Calculations are 2-local unless otherwise stated.

\subsection{First Stable Homotopy Groups of \texorpdfstring{$\Sigma\bb{C}P^{\infty}_2$}{CPinfinity2}} 
We compute $\pi_*(Q\Sigma\bb{C}P^{\infty}_2) = \pi_*^s(\Sigma\bb{C}P^{\infty}_2)$ through a range.
The mod two cohomology of $\Sigma\bb{C}P^{\infty}_2$
$$H^*(\Sigma\bb{C}P^{\infty}_2;\bb{Z}/2) \cong \bb{Z}/2 \cdot \{y_5, y_7, \cdots, y_{2n+1}, \cdots\}$$
has a single $\bb{Z}/2$-generator in every odd degree $i\geq 5$, which we denote by the $y_i$. There are no nontrivial cup products.
The class $y_{2n+1}$ can be identified with the suspension of the class $x^n$ in the cohomology of $\bb{C}P^{\infty}$, and the action of the Steenrod squares is then identified with that on $\bb{C}P^\infty$.
The diagram below exhibits the action of the Steenrod squares on elements of $H^*(\Sigma\bb{C}P^{\infty}_2;\bb{Z}/2)$ in low degrees.
Here a straight line segment represent a nontrivial action of $\Sq^2$, a curved line segment represent that of $\Sq^4$, and a dashed curved line segment represent that of $\Sq^8$.

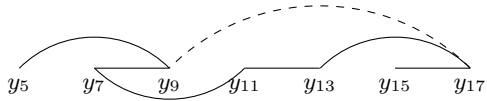
\begin{center}
  \begin{tikzpicture}
	 \draw (0, 0) node[below]{{\footnotesize{$y_5$}}};
			\draw (1,0) node[below]{{\footnotesize{$y_7$}}} -- (2,0) node[below]{{\footnotesize{$y_9$}}}; \draw [thin] (3,0) node[below]{{\footnotesize{$y_{11}$}}} -- (4,0) node[below]{{\footnotesize{$y_{13}$}}}; \draw [thin] (5,0) node[below]{{\footnotesize{$y_{15}$}}} -- (6,0) node[below]{{\footnotesize{$y_{17}$}}};
	  \draw (2,0) arc (45:135:1.41cm); \draw (3,0) arc (-45:-135:1.41cm); \draw (6,0) arc (45:135:1.41cm);
		\draw [dashed] (6,0) arc (45:135:2.83cm);
	\end{tikzpicture}
\end{center}

We now apply the Adams spectral sequence to compute the (2-local part of) the stable homotopy of $\Sigma\bb{C}P^{\infty}_2$, which has
$$E_2^{s,t} = \ext^{s, t}_{\cal{A}} \big{(}H^*(\Sigma\bb{C}P^{\infty}_2;\bb{Z}/2), \bb{Z}/2\big{)} \Longrightarrow \pi_{t-s}^{s}(\Sigma\bb{C}P^{\infty}_2).$$

In the Appendix we present an explicit minimal $\cal{A}$-resolution of $H^*(\Sigma\bb{C}P^{\infty}_2;\bb{Z}/2)$ to compute these Ext groups. 
We summarize the result as follows.

\begin{center}
\DeclareSseqGroup\tower {} {
    \class(0,0)
    \foreach \y in {1,...,9} {
        \class(0,\y)
\structline
}}

\begin{sseqdata}[name = first results, Adams grading, yscale = 0.7, xscale = 1.2, y range = {0}{4}, x range = {5}{15}]
\tower(5, 0)
\class(6, 1)
\class(7, 2)
\tower(7, 0)
\class(8, 1) 
\structline(5, 0)(6, 1)
\structline(6, 1)(7, 2)
\structline(7, 0)(8, 1)
\structline(5, 0)(8, 1)

\class(15, 0)
\class(12, 1)
\class(11, 2)
\class(9, 3)
\class(9, 4)
\structline(9, 3)(9, 4)

\class["?"](13, 1) \class["?"](14, 1) \class["?"](15, 1)
\class["?"](12, 2) \class["?"](13, 2) \class["?"](14, 2) \class["?"](15, 2)
\class["?"](11, 3) \class["?"](12, 3) \class["?"](13, 3) \class["?"](14, 3) \class["?"](15, 3)
\class["?"](10, 4) \class["?"](11, 4) \class["?"](12, 4) \class["?"](13, 4) \class["?"](14, 4) \class["?"](15, 4)

\end{sseqdata}
\printpage[name =  first results, title = The 2-primary Adams $E_2$-page for $\pi_*^s(\Sigma\bb{C}P^{\infty}_2)$]
\end{center}

\vspace{2mm}

As usual, the horizontal axis is $t-s$ and the vertical axis is $s$. Each small circled dot represents a $\bb{Z}/2$. 
Bigger circled dots marked by question marks denote unknown groups. Places without circled dots are all zero.
As usual, vertical line segments represent multiplication by $h_0$, line segments of slope one represent multiplication by $h_1$, and line segments of slope 1/3 represent multiplication by $h_2$, etc. There is a single $h_0$ tower along colunms $t-s = 5$ and $t-s = 7$. There is nothing in the chart whenever $t-s \leq 4$, as $\Sigma\bb{C}P^{\infty}_2$ is 4-connected.

Any  $d_r$-differential starting from or arriving at columns with $t-s\leq 8$, $\forall r\geq 2$
must be trivial. One can then read off some first few stable homotopy groups of $\Sigma\bb{C}P^{\infty}_2$, which we summarize as follows.

\begin{lemma} \label{sthomotopy}
The stable homotopy groups $\pi_i^s(\Sigma\bb{C}P^{\infty}_2)$ for $i \leq 8$ are as follows.
\begin{center}
\begin{tabular}{c|c|c|c|c|c}
\hline
$i$ & $\leq 4$ & $5$ & $6$ & $7$ & $8$  \\
\hline
$\pi_i^s(\Sigma\bb{C}P^{\infty}_2)$ & $0$ & $\bb{Z}$ & $\bb{Z}/2$ & $\bb{Z} \oplus \bb{Z}/2$ & $\bb{Z}/2$ \\
\hline
\end{tabular}
\end{center}
\end{lemma}

\vspace{1mm}

\subsection{Recovery of Corollary \ref{cor3}} \label{scor3}
We now make use of Lemma \ref{sthomotopy} to compute $\{\bb{C}P^3, \Sigma\bb{C}P^{\infty}_2\}$. 
To start, consider the following long exact sequence
$$\cdots \rightarrow \{S^4, \Sigma\bb{C}P^{\infty}_2\} \longrightarrow \{S^5, \Sigma\bb{C}P^{\infty}_2\} \longrightarrow \{\Sigma \bb{C}P^2, \Sigma\bb{C}P^{\infty}_2\} \longrightarrow \{S^3, \Sigma\bb{C}P^{\infty}_2\} \rightarrow \cdots.$$
associated to the cofiber sequence
$S^3\stackrel{\eta}{\rightarrow} S^2 \rightarrow \bb{C}P^2$ (where $\eta$ denotes the Hopf map).
According to Lemma \ref{sthomotopy} the groups $\{S^3, \Sigma\bb{C}P^{\infty}_2\}$ and $\{S^4, \Sigma\bb{C}P^{\infty}_2\}$ both vanish. Therefore the middle two groups in the above sequence are isomorphic, and by Lemma \ref{sthomotopy} we obtain that
$$\{\Sigma \bb{C}P^2, \Sigma\bb{C}P^{\infty}_2\} \cong \{S^5, \Sigma\bb{C}P^{\infty}_2\} = \pi_5^s(\Sigma\bb{C}P^{\infty}_2) \cong \bb{Z}.$$

Next we consider the cofiber sequence $S^5\stackrel{\eta_3}{\rightarrow} \bb{C}P^2 \rightarrow \bb{C}P^3$, where $\eta_3$ denotes the attaching map of the top cell of $\bb{C}P^3$. In the following diagram
\begin{center}
\begin{tikzcd}
 \cdots \ar[r]  & \{\Sigma \bb{C}P^2, \Sigma\bb{C}P^{\infty}_2 \} \ar[r, "(\Sigma \eta_3)^*"] & \{S^6, \Sigma\bb{C}P^{\infty}_2 \} \ar[r] & \{ \bb{C}P^3, \Sigma\bb{C}P^{\infty}_2 \} \ar[r] & 0 \\
  & \{ S^5, \Sigma\bb{C}P^{\infty}_2 \} \cong \bb{Z} \ar[u, "\cong"] \ar[ur] & & & 
\end{tikzcd}
\end{center}
the top row is part of the long exact sequence associated with a cofiber sequence, and the vertical map in the triangle is the isomorphism we just analyzed. 
Let $q_i: \bb{C}P^i \rightarrow S^{2i}$ be the quotient map by the $(2i-1)$-skeleton.
By Lemma \ref{sthomotopy} we have $\{S^6, \Sigma\bb{C}P^{\infty}_2\} \cong \pi_6^s(\Sigma\bb{C}P^{\infty}_2) \cong \bb{Z}/2$.
Thus 
$$(\Sigma \eta_3)^*: \{\Sigma \bb{C}P^2, \Sigma\bb{C}P^{\infty}_2\} \longrightarrow \{S^6, \Sigma\bb{C}P^{\infty}_2\}$$ 
is seen to be a homomorphism $\bb{Z}\rightarrow \bb{Z}/2$, and $\{\bb{C}P^3, \Sigma\bb{C}P^{\infty}_2\}$ is the cokernel of this homomorphism. 
We claim that $(\Sigma \eta_3)^*: \bb{Z}\rightarrow \bb{Z}/2$ is the zero homomorphism. 
Note that the homomorphism $\{S^5, \Sigma\bb{C}P^{\infty}_2\}\rightarrow \{S^6, \Sigma\bb{C}P^{\infty}_2\}$ in the above diagram is induced by the composite
$$\big{(} S^6\stackrel{\Sigma \eta_3}{\longrightarrow} \Sigma\bb{C}P^2 \stackrel{\Sigma q_2}{\longrightarrow} S^5 \big{)} = \Sigma \big{(} S^5\stackrel{\eta_3}{\longrightarrow} \bb{C}P^2 \stackrel{q_2}{\longrightarrow} S^4 \big{)}.$$
Here we recall a general fact. Let $\eta_l: S^{2l-1}\rightarrow \bb{C}P^{l-1}$ denote the attaching map of the top cell of $\bb{C}P^l$.
Then the composite $S^{2l-1}\stackrel{\eta_l}{\longrightarrow} \bb{C}P^{l-1} \stackrel{q_{l-1}}{\longrightarrow} S^{2l-2}$ is the suspension of the Hopf map $S^3\stackrel{\eta}{\longrightarrow} S^2$ if $l$ is even, and is null when $l$ is odd. 
(This fact can be seen as a description of the structure of the stunted projective spaces $\bb{C}P^l_{l-1} = \bb{C}P^l/\bb{C}P^{l-2}$, $l\geq 3$. Indeed, it is detected by the action of $Sq^2$ that this two-cell complex is equivalent to $S^{2l}\vee S^{2l-2}$ if $l$ is odd, and is equivalent to $\Sigma^{2l-4}\bb{C}P^2$ when $l$ is even.) It follows that the composite 
$$\bb{Z}\cong \{S^5, \Sigma\bb{C}P^{\infty}_2\} \stackrel{\cong}{\longrightarrow} \{\Sigma \bb{C}P^2, \Sigma\bb{C}P^{\infty}_2\} \stackrel{(\Sigma \eta_3)^*}{\longrightarrow} \{S^6, \Sigma\bb{C}P^{\infty}_2\} \cong \bb{Z}/2$$
is induced by a null map, and is hence the zero homomorphism. This completes the proof that 
$\{\bb{C}P^3, \Sigma\bb{C}P^{\infty}_2\}\cong \bb{Z}/2$, and hence that of Corollary \ref{cor3}.

\subsection{Recovery of Corollary \ref{cor4}} \label{scor4}
We now show that 
$\{\bb{C}P^4, \Sigma\bb{C}P^{\infty}_2\} = 0,$  which recovers Corollary \ref{cor4}.
Rather than using long exact sequences associated to cofibration sequences,
 we organize the calculation with the Atiyah-Hirzebruch spectral sequence. 
Recall that $\{X, \Sigma^t Y\}$ can be regarded as a generalized cohomology theory for $X$.
Filtering by skeleta leads to the Atiyah-Hirzebruch Spectral Sequence (AHSS) for these stable maps with 
$E_2^{p, q} = H^p\big{(}X; \pi_{-q}(Y) \big{)}$,
which converges to $\{X, \Sigma^{p+q} Y\}$ if $X$ is a finite complex.
The AHSS computing $\{\bb{C}P^4, \Sigma\bb{C}P^{\infty}_2\}$ thus has
$$E_2^{p, q} = H^p\big{(}\bb{C}P^4; \pi_{-q}^s(\Sigma\bb{C}P^{\infty}_2)\big{)} \Longrightarrow \{X, \Sigma^{p+q+1} \bb{C}P^{\infty}_2 \}.$$
We show the $E_2$ term through a range below. As $E_2^{p, q}$ vanishes for all $p>8$ and for all $q>-5$,
 only the circled groups can contribute in total degree zero.

\begin{center}
\begin{tikzpicture}
  \matrix (m) [matrix of math nodes, nodes in empty cells, nodes={minimum width=5ex, minimum height=5ex, outer sep=-5pt}, column sep=1ex,row sep=1ex]{
        \quad\strut     &    &   2   & 3 &  4 & 5  & 6 & 7 & 8   & \geq 9 & & \strut  \\
				  \geq -4  &    &  0        & 0 & 0      & 0     &  0        & 0    & 0      & 0 & & \\
          -5       &    &  \bb{Z}   & 0 & \bb{Z} & 0 &  \bb{Z}   & 0    & \bb{Z} & 0 & & \\
          -6       &    &  \bb{Z}/2 & 0 &  \bb{Z}/2  & 0 & \ZZ & 0 & \bb{Z}/2 & 0 & & \\
					-7       &    &  \bb{Z}\oplus \bb{Z}/2 & 0 &  \bb{Z}\oplus\bb{Z}/2 & 0 & \bb{Z}\oplus\bb{Z}/2 & 0 & \bb{Z}\oplus\bb{Z}/2 & 0 & & \\
					-8       &    &  \bb{Z}/2 & 0 &  \bb{Z}/2 & 0 & \bb{Z}/2 & 0 & \ZZ & 0 & & \\
					};
					\draw[thick] (m-1-2.east) -- (m-6-2.east);
					\draw[thick] (m-1-1.south) -- (m-1-11.south);
					\draw[dashed,->] (m-3-5) -- (m-4-7);
					\draw[dashed,->] (m-5-7) -- (m-6-9);
\end{tikzpicture}
\end{center}

We analyze the relevant $d_2$ differentials, pictured above. By construction of the AHSS, the differential
$d_2: E_2^{2a, -b}\rightarrow E_2^{2a+2, -b-1}$ is 
induced by the connecting map $\bb{C}P^{a+1}/\bb{C}P^a \rightarrow \Sigma \bb{C}P^a/\bb{C}P^{a-1}$ in the cofiber sequence
$$\bb{C}P^{a}/\bb{C}P^{a-1} \rightarrow \bb{C}P^{a+1}/\bb{C}P^{a-1} \rightarrow \bb{C}P^{a+1}/\bb{C}P^a \rightarrow \Sigma \bb{C}P^a/\bb{C}P^{a-1}.$$
This connecting map, $S^{2a+2}\rightarrow S^{2a+1}$, both reflects and is determined by the structure of the two-cell stunted projective space $\bb{C}P^{a+1}/\bb{C}P^{a-1}$, and is detected by $\Sq^2$.
As  discussed in Section~\ref{scor3}, when $a$ is odd this  connecting map is
a suspension of $\eta: S^3\rightarrow S^2$, and when $a$ is even it is null.
Furthermore,
$$E_2^{2a, -b} = \pi_b^s(\Sigma\bb{C}P^{\infty}_2), \text{ and } E_2^{2a+2, -b-1} = \pi_{b+1}^s(\Sigma\bb{C}P^{\infty}_2).$$
In summary, $d_2: E_2^{2a, -b} \rightarrow E_2^{2a+2, -b-1}$ is a homomorphism $\pi_b^s(\Sigma\bb{C}P^{\infty}_2) \rightarrow \pi_{b+1}^s(\Sigma\bb{C}P^{\infty}_2)$, which is multiplication by $\eta$ when $a$ is odd, and is zero when $a$ is even.

For example, taking $a=3$ and $b=6$, one concludes that $E_2^{6, -6}\rightarrow E_2^{8, -7}$ is the homomorphism $\bb{Z}\rightarrow \bb{Z}\oplus \bb{Z}/2$ induced by $\eta$. This homomorphism is onto the second summand of $\bb{Z}\oplus \bb{Z}_2$, due to the fact that on the Adams $E_2$-page the dot at $(5, 0)$ is connected with that at $(6, 1)$ via multiplication by $h_1$.
Similarly, one deduces that $E_2^{6, -7}\rightarrow E_2^{8, -8}$ is the homomorphism $\bb{Z}\oplus \bb{Z}/2 \rightarrow \bb{Z}/2$ which is the surjection when restricted to the first summand and is zero when restricted to the second. 

We conclude that $E_r^{p, -p} = 0$ for all $p$ and all $r\geq 3$. So $\{\bb{C}P^4, \Sigma\bb{C}P^{\infty}_2\} = 0$, and Corollary \ref{cor4} is proved.

\section{Proof of Main Theorems}

We prove the main results of the paper, Theorems \ref{firstunstable} and \ref{secondunstable}, which give
 vanishing Chern enumeration for bundles of rank $(l-1)$ and $(l-2)$ over $\bb{C}P^l$. Both these cases are in the metastable range, so by Theorem \ref{reduction} there are identifications
$$\Vect_{l-1}^0(\bb{C}P^l) \cong \{\bb{C}P^l, \Sigma \bb{C}P^{\infty}_{l-1} \} \quad \text{and} \quad \Vect_{l-2}^0(\bb{C}P^l) \cong \{\bb{C}P^l, \Sigma\bb{C}P^{\infty}_{l-2} \}.$$

We perform these stable calculations, following the strategy of Section~\ref{scor4} to regard them as generalized cohomology computations.

\subsection{First Stable Homotopy Groups of \texorpdfstring{$\Sigma\bb{C}P^{\infty}_n$}{CPinfinityn}}

Both the calculations of $\{\bb{C}P^l, \Sigma\bb{C}P^{\infty}_{l-1} \}$ and $\{\bb{C}P^l, \Sigma \bb{C}P^{\infty}_{l-2} \}$ rely on the knowledge of some first few 2-primary stable homotopy groups of $\Sigma\bb{C}P^{\infty}_n$. The latter also requires knowledge of the 3-primary groups.
This subsection is dedicated to presenting the results of these calculations, with  details postponed to the Appendix.
 
The cohomology of $\Sigma\bb{C}P^{\infty}_n$ has a single $\bb{Z}/2$-generator $y_{2n+2k+1}$ in every odd degree greater than or equal to $2n+1$. 
The class $y_{2n+2k+1}$ can be identified with the suspension of the class $x^{n+k}$ in the cohomology of $\bb{C}P^{\infty}$, to compute the action of the Steenrod squares. 
The stable homotopy groups $\pi_i^s(\Sigma\bb{C}P^{\infty}_n)$ within the range $i\leq 2n+4$ require only the information of the Steenrod square actions on the finite skeleton $\Sigma \bb{C}P^{n+4}_n$ of $\Sigma \bb{C}P^{\infty}_n$. 
These actions, and hence the Adams $E_2$-pages, exhibit an 8-fold periodic behavior. We obtain the following result, to be proved in the Appendix.

\begin{lemma} \label{sthomotopygeneral} The 2-primary stable homotopy groups $\pi_i^s(\Sigma\bb{C}P^{\infty}_n)$ for $i \leq 2n+4$ can be described as follows.
\begin{enumerate}
	\item $\pi_i^s(\Sigma\bb{C}P^{\infty}_n) = 0$ whenever $i\leq 2n$, and $\pi_{2n+1}^s(\Sigma\bb{C}P^{\infty}_n) = \bb{Z}$.
	\item $\pi_{2n+2}^s(\Sigma\bb{C}P^{\infty}_n) = 0$ if $n$ is odd, and $\pi_{2n+2}^s(\Sigma\bb{C}P^{\infty}_n) = \bb{Z}/2$ if $n$ is even.
	\item $\pi_{2n+3}^s(\Sigma\bb{C}P^{\infty}_n) = \bb{Z}$ if $n$ is odd, and $\pi_{2n+3}^s(\Sigma\bb{C}P^{\infty}_n) = \bb{Z}\oplus \bb{Z}/2$ if $n$ is even.
	\item $\pi_{2n+4}^s(\Sigma\bb{C}P^{\infty}_n)$ exhibits the following 8-fold periodicity.
	    
\begin{center}
\begin{tabular}{c|c|c|c|c|c|c|c|c}
\hline
$n \mod 8$ & $0$ & $1$ & $2$ & $3$ & $4$ & $5$ & $6$ & $7$ \\
\hline
$\pi_{2n+4}^s(\Sigma\bb{C}P^{\infty}_n)$ & $\bb{Z}/8$ & $\bb{Z}/4$ & $\bb{Z}/2$ & $0$ & $\bb{Z}/4$ & $\bb{Z}/2$ & $\bb{Z}/2$ & $0$ \\
\hline
\end{tabular}
\end{center}
	
\end{enumerate}
\end{lemma}

\subsection{The First Unstable Case}
We compute $\{\bb{C}P^l, \Sigma\bb{C}P^{\infty}_{l-1} \}$ through analysis of the
 AHSS, which has 
$$E_2^{p, q} = H^p\big{(}\bb{C}P^l; \pi_{-q}^s(\Sigma\bb{C}P^{\infty}_{l-1})\big{)} \Longrightarrow \{ \bb{C}P^l, \Sigma^{p+q+1} \bb{C}P^{\infty}_{l-1} \}.$$

When $l$ is even,  $\pi_{2l-1}^s(\Sigma\bb{C}P^{\infty}_{l-1}) = \bb{Z}$ and $\pi_{2l}^s(\Sigma\bb{C}P^{\infty}_{l-1}) = 0$ by Lemma \ref{sthomotopygeneral}, and therefore the $E_2$-page has the following form:

\begin{center}
\begin{tikzpicture}
  \matrix (m) [matrix of math nodes, nodes in empty cells, nodes={minimum width=5ex, minimum height=5ex, outer sep=-5pt}, column sep=1ex,row sep=1ex]{
        \quad\strut     &    &   2   & 3 &  4   &   \cdots  & \cdots  & 2l-2 & 2l-1 & 2l  & \geq 2l+1 &  &\strut  \\
					\geq -(2l-2)  &    &  0   & 0 &  0  &   \cdots  & \cdots  & 0   & 0    & 0 & 0 &  & \\
          -(2l-1)       &    &  \bb{Z}   & 0 &  \bb{Z}  &   \cdots  & \cdots  & \bb{Z}   & 0    & \bb{Z} & 0 & & \\
          -2l           &    &  0 & 0 &  0  &  \cdots & \cdots  & 0 & 0 & 0 & 0 &  &\\};
					\draw[thick] (m-1-2.east) -- (m-4-2.east);
					\draw[thick] (m-1-1.south) -- (m-1-12.south);
\end{tikzpicture}.
\end{center}

The terms $E_2^{p, q}$ vanish for all $p>2l$ and for all $q>-(2l-1)$. Since all groups are zero along the diagonal $p+q = 0$, one concludes immediately that 
$\{\bb{C}P^l, \Sigma\bb{C}P^{\infty}_{l-1} \} = 0$. 

When $l$ is odd, $\pi_{2l-1}^s(\Sigma\bb{C}P^{\infty}_{l-1}) = \bb{Z}$ and $\pi_{2l}^s(\Sigma\bb{C}P^{\infty}_{l-1}) = \bb{Z}/2$ by Lemma \ref{sthomotopygeneral}, and therefore the $E_2$-page has the following form:

\begin{center}
\begin{tikzpicture}
  \matrix (m) [matrix of math nodes, nodes in empty cells, nodes={minimum width=5ex, minimum height=5ex, outer sep=-5pt}, column sep=1ex,row sep=1ex]{
        \quad\strut     &    &   2   & 3 &  4   &   \cdots  & \cdots  & 2l-2 & 2l-1 & 2l & \geq 2l+1  &  &\strut  \\
				  \geq -(2l-2)  &    &  0   & 0 &  0  &   \cdots  & \cdots  & 0   & 0    & 0 & 0 & & \\
          -(2l-1)       &    &  \bb{Z}   & 0 &  \bb{Z}  &   \cdots  & \cdots  & \bb{Z}   & 0    & \bb{Z} & 0 & & \\
          -2l           &    &  \bb{Z}/2 & 0 &  \bb{Z}/2  &  \cdots & \cdots  & \bb{Z}/2 & 0 & \ZZ & 0 & &\\};
					\draw[thick] (m-1-2.east) -- (m-4-2.east);
					\draw[thick] (m-1-1.south) -- (m-1-12.south);
					\draw[dashed,->] (m-3-8) -- (m-4-10);
\end{tikzpicture}.
\end{center}

We thus determine the differential $d_2: \bb{Z} = E_2^{2l-2, -(2l-1)} \rightarrow E_2^{2l, -2l} = \bb{Z}/2$.
As was previously analyzed, this is a homomorphism $\pi_{2l-1}^s(\Sigma \bb{C}P^{\infty}_{l-1}) \rightarrow \pi_{2l}^s(\Sigma \bb{C}P^{\infty}_{l-1})$
induced by a map $S^{2l} \rightarrow S^{2l-1}$, which is a suspension of $\eta: S^3\rightarrow S^2$ when $l$ is even, and null when $l$ is odd.
Thus the above differential is the zero as $l$ is odd, and the circled group $\bb{Z}/2$ survives to the infinity page. 
Since this is the only nonzero group along the diagonal $p+q=0$ on that page, 
we conclude that $\{\bb{C}P^l, \Sigma\bb{C}P^{\infty}_{l-1} \} = \bb{Z}/2$.

We have proved that $\{\bb{C}P^l, \Sigma\bb{C}P^{\infty}_{l-1} \} = 0$ when $l$ is even, and $\{\bb{C}P^l, \Sigma\bb{C}P^{\infty}_{l-1} \} = \bb{Z}/2$ when $l$ is odd, completing the proof of Theorem \ref{firstunstable}.

\subsection{The Second Unstable Case}
We now prove Theorem \ref{secondunstable}, by calculating $\{\bb{C}P^l, \Sigma\bb{C}P^{\infty}_{l-2} \}$. In this case both the prime 2 and the prime 3 are involved. 
We  carry out the 2-local calculation in Subsection 4.3.1, and the 3-local computation in Subsection 4.3.2. The results are as follows.

\begin{theorem} \label{2local}
Let $l\geq 4$ be an integer. Then $\{\bb{C}P^l, \Sigma\bb{C}P^{\infty}_{l-2} \}_{(2)}$ exhibits the following 8-fold periodic behavior.
\begin{center}
\begin{tabular}{c | c | c | c | c | c | c | c | c }
\hline
$l \mod 8$ & $0$ & $1$ & $2$ & $3$ & $4$ & $5$ & $6$ & $7$ \\
\hline
$\{\bb{C}P^l, \Sigma\bb{C}P^{\infty}_{l-2} \}_{(2)}$ & $0$ & $0$ & $\bb{Z}/4$ & $\bb{Z}/2$ & $0$ & $0$ & $\bb{Z}/2$ & $\bb{Z}/2$ \\
\hline
\end{tabular}
\end{center}
\end{theorem}

\begin{theorem} \label{3local}
Let $l\geq 4$ be an integer. Then $\{\bb{C}P^l, \Sigma\bb{C}P^{\infty}_{l-2} \}_{(3)}$ exhibits the following 3-fold periodic behavior.
\begin{enumerate}
	\item The group vanishes whenever $l$ is $0$ or $1$ mod $3$.
	\item The group is isomorphic to $\bb{Z}/3$ when $l$ is $2$ mod 3.
\end{enumerate}
\end{theorem}

Moreover, $\{\bb{C}P^l, \Sigma\bb{C}P^{\infty}_{l-2} \}$ has no $p$-torsion for $p\geq 5$. Combining Theorems \ref{2local} and \ref{3local}, one obtains immediately the enumerations in Theorem \ref{secondunstable}.

\subsubsection{Calculations at the prime 2}
We prove Theorem \ref{2local}, with 2-local computations throughout.
We regard $\{\bb{C}P^l, \Sigma\bb{C}P^{\infty}_{l-2} \}$ again as a cohomology calculation and apply the AHSS,
which has
$$E_2^{p, q} = H^p\big{(}\bb{C}P^l; \pi_{-q}^s(\Sigma\bb{C}P^{\infty}_{l-2})\big{)} \Longrightarrow \{ \bb{C}P^l, \Sigma^{p+q+1} \bb{C}P^{\infty}_{l-2} \}.$$

By Lemma \ref{sthomotopygeneral}, stable homotopy groups of $\Sigma\bb{C}P^{\infty}_{l-2}$ can be summarized as follows.

\begin{center}
\begin{tabular}{c | c | c | c | c | c }
\hline
   & $\pi_{< 2l-3}^s(\Sigma\bb{C}P^{\infty}_{l-2})$ & $\pi_{2l-3}^s(\Sigma\bb{C}P^{\infty}_{l-2})$ & $\pi_{2l-2}^s(\Sigma\bb{C}P^{\infty}_{l-2})$ & $\pi_{2l-1}^s(\Sigma\bb{C}P^{\infty}_{l-2})$ & $\pi_{2l}^s(\Sigma\bb{C}P^{\infty}_{l-2})$  \\
\hline
  $l=0$ mod 8 & 0 & $\bb{Z}$ & $\bb{Z}/2$ & $\bb{Z}\oplus \bb{Z}/2$ & $\bb{Z}/2$ \\
	$l=1$ mod 8 & 0 & $\bb{Z}$ & $0$ & $\bb{Z}$ & $0$ \\
	$l=2$ mod 8 & 0 & $\bb{Z}$ & $\bb{Z}/2$ & $\bb{Z}\oplus \bb{Z}/2$ & $\bb{Z}/8$ \\
  $l=3$ mod 8 & 0 & $\bb{Z}$ & $0$ & $\bb{Z}$ & $\bb{Z}/4$ \\
	$l=4$ mod 8 & 0 & $\bb{Z}$ & $\bb{Z}/2$ & $\bb{Z}\oplus \bb{Z}/2$ & $\bb{Z}/2$ \\
	$l=5$ mod 8 & 0 & $\bb{Z}$ & $0$ & $\bb{Z}$ & $0$ \\
	$l=6$ mod 8 & 0 & $\bb{Z}$ & $\bb{Z}/2$ & $\bb{Z}\oplus \bb{Z}/2$ & $\bb{Z}/4$ \\
	$l=7$ mod 8 & 0 & $\bb{Z}$ & $0$ & $\bb{Z}$ & $\bb{Z}/2$ \\
\hline
\end{tabular}
\end{center}

Therefore the AHSS also exhibits an 8-fold periodic behavior. 
Since the proofs of these eight cases are similar, we shall present a proof for one case and omit the details for the others. 
Let us consider the case $l = 2$ mod 8. In this case our goal is to show that, 2-locally, $\{\bb{C}P^l, \Sigma\bb{C}P^{\infty}_{l-2}\} \cong \bb{Z}/4$.

Part of the $E_2$-page of the spectral sequence is presented below. The groups $E_2^{p, q}$ vanish for all $p>2l$ and for all $q>-(2l-3)$, so only the circled groups can contribute in total degree zero.

\begin{center}
\begin{tikzpicture}
  \matrix (m) [matrix of math nodes, nodes in empty cells, nodes={minimum width=5ex, minimum height=5ex, outer sep=-5pt}, column sep=1ex,row sep=1ex]{
        \quad\strut     &    &   \cdots  &  2l-4 & 2l-3  & 2l-2 & 2l-1 & 2l   & \geq 2l+1 &  \strut  \\
				  \geq -(2l-4)  &    &    \cdots   & 0      & 0     & 0        & 0    & 0      & 0 & \\
          -(2l-3)       &    &    \cdots  & \bb{Z} & 0 & \bb{Z}   & 0    & \bb{Z} & 0 &  \\
          -(2l-2)       &    &    \cdots  & \bb{Z}/2 & 0 & \ZZ & 0 & \bb{Z}/2 & 0 &  \\
					-(2l-1)       &    &   \cdots & \bb{Z}\oplus\bb{Z}/2 & 0 & \bb{Z}\oplus\bb{Z}/2 & 0 & \bb{Z}\oplus\bb{Z}/2 & 0 &  \\
					-2l           &    &    \cdots & \bb{Z}/8 & 0 & \bb{Z}/8 & 0 & \ZZZ & 0 &  \\
					};
					\draw[thick] (m-1-2.east) -- (m-6-2.east);
					\draw[thick] (m-1-1.south) -- (m-1-10.south);
					\draw[dashed,->] (m-3-4) -- (m-4-6);
					\draw[dashed,->] (m-3-6) -- (m-4-8);
					\draw[dashed,->] (m-4-6) -- (m-5-8);
					\draw[dashed,->] (m-5-6) -- (m-6-8);
\end{tikzpicture}
\end{center}

We first analyze the relevant $d_2$-differentials. 
As was discussed in Section~\ref{scor4}, they are determined by the structures of certain two-cell stunted projective spaces, where the attaching maps are detected by $\Sq^2$. 
For example, for each $q$ the differential $d_2^{2l-2, q}: E_2^{2l-2, q} \longrightarrow E_2^{2l, q-1}$ is a homomorphism induced by the map 
$\bb{C}P^l/\bb{C}P^{l-1} \rightarrow \Sigma \bb{C}P^{l-1}/\bb{C}P^{l-2}$, which is part of the cofiber sequence defining $\bb{C}P^l/\bb{C}P^{l-2}$.
This map $S^{2l}\rightarrow S^{2l-1}$ is a suspension of $\eta: S^3\rightarrow S^2$ since our $l$ is even. 
Similarly, by the structure of $\bb{C}P^{l-1}/\bb{C}P^{l-3}$, the differentials $d_2^{2l-4, q}: E_2^{2l-4, q} \longrightarrow E_2^{2l-2, q-1}$ are all induced by the null map, and are therefore all zero.

The Adams spectral sequence for the stable homotopy of $\Sigma\bb{C}P^{\infty}_{l-2}$ has the following form. (See the Appendix for details.)

\begin{center}
\DeclareSseqGroup\tower {} {
    \class(0,0)
    \foreach \y in {1,...,9} {
        \class(0,\y)
\structline
}}

\begin{sseqdata}[name = example, Adams grading, yscale = 0.7, xscale = 1.6, y range = {0}{4}, x range = {4}{10}, no x ticks]
\begin{scope}[ background ]
\node at (5,\ymin - 1) {2l-3};
\node at (6,\ymin - 1) {2l-2};
\node at (7,\ymin - 1) {2l-1};
\node at (8,\ymin - 1) {2l};
\node at (9,\ymin - 1) {2l+1};
\end{scope}
\tower(5, 0)
\class(6, 1)
\class(7, 2)
\tower(7, 0)
\class(8, 1) \class(8, 2) \class(8, 3)
\structline(5, 0)(6, 1)
\structline(6, 1)(7, 2)
\structline(7, 2)(8, 3)
\structline(8, 1)(8, 2)
\structline(8, 2)(8, 3)
\tower(9, 1)

\end{sseqdata}
\printpage[name = example, title = The 2-primary Adams $E_2$-page for $\pi_*^s(\Sigma\bb{C}P^{\infty}_{l-2})$ when $l\equiv 2$ mod 8]
\end{center}

When $l=2$ mod 8, we use similar analysis as previous to conclude the following about the $d_2$ differentials:
\begin{itemize}
	\item $d_2^{2l-2, -(2l-3)}$ is the surjection $\bb{Z}\rightarrow \bb{Z}/2$;
	\item $d_2^{2l-2, -(2l-2)}: \bb{Z}_2 \rightarrow \bb{Z}\oplus \bb{Z}_2$ is the inclusion into the second summand;
	\item $d_2^{2l-2, -(2l-1)}: \bb{Z}\oplus \bb{Z}_2 \rightarrow \bb{Z}/8$ is given by $(0, 4)$;
	\item $d_2^{2l-4, q} = 0$ for all $q$.
\end{itemize}

For example, to determine $d_2^{2l-2, -(2l-1)}: \bb{Z}\oplus \bb{Z}_2 \rightarrow \bb{Z}/8$ in the third case above, we first note that it is a homomorphism
$\pi_{2l-1}^s(\Sigma \bb{C}P^{\infty}_{l-2}) \rightarrow \pi_{2l}^s(\Sigma \bb{C}P^{\infty}_{l-2})$ given by multiplication by $\eta$, where $\eta$ now denotes the stable element. We then examine the Adams chart above to notice that the dot at $(2l-1, 2)$, which yields the $\bb{Z}/2$-summand of the domain, is connected with the dot at $(2l, 3)$ by $h_1$, which represents the element 4 in the target group. This implies that $d_2^{2l-2, -(2l-1)}$ restricted to the $\bb{Z}/2$-summand is multiplication by 4. Similarly, we conclude that $d_2^{2l-2, -(2l-1)}$ restricted to the $\bb{Z}$-summand must be the zero homomorphism, since the dots at $(2l-1, 0)$ and $(2l, 1)$ are not connected by $h_1$.

For degree reasons there cannot be any $d_3$-differential in the AHSS, and the $E_4$-page is partly depicted below.
Here along the diagonal line $p+q = 0$ there is a single group of $E_4^{2l, -2l} = \bb{Z}/4$, and it follows immediately that $\{\bb{C}P^l, \Sigma\bb{C}P^{\infty}_{l-2}\}$ is a quotient of $\bb{Z}/4$, by the image of the possibly nontrivial $d_4$ differential $d_4^{2l-4, -(2l-3)}: E_4^{2l-4, -(2l-3)} \rightarrow E_4^{2l, -2l}$.

\begin{center}
\begin{tikzpicture}
  \matrix (m) [matrix of math nodes, nodes in empty cells, nodes={minimum width=5ex, minimum height=5ex, outer sep=-5pt}, column sep=1ex,row sep=1ex]{
        \quad\strut  & &   \cdots  &  2l-4 & 2l-3  & 2l-2 & 2l-1 & 2l & \geq 2l+1  &  \strut  \\
		   \geq -(2l-4)  & &    \cdots   & 0      & 0     & 0        & 0    & 0      & 0 & & \\
          -(2l-3)    & &   \cdots  & \bb{Z} & 0 & 2\bb{Z}   & 0    & \bb{Z} & 0 &  \\
          -(2l-2)    & &   \cdots  & * & 0 & 0 & 0 & 0 & 0 &  \\
					-(2l-1)    & &   \cdots  & * & 0 & \bb{Z} & 0 & \bb{Z} & 0 & \\
					-2l        & &   \cdots  & * & 0 & * & 0 & \ZZZZZ & 0 &  \\
					};
					\draw[dashed, ->] (m-3-4.south east) -- (m-6-8.north west);
					\draw[thick] (m-1-2.east) -- (m-6-2.east);
					\draw[thick] (m-1-1.south) -- (m-1-10.south);
\end{tikzpicture}
\end{center}

By the construction of the AHSS, the $d_4$ differentials is determined by the structure of the three-cell complex $\bb{C}P^l/\bb{C}P^{l-3}$. 
The differential $d_4^{2l-4, -(2l-3)}: E_4^{2l-4, -(2l-3)}\rightarrow E_4^{2l, -2l}$ is induced by a map 
$\lambda_l: S^{2l} = \bb{C}P^l/\bb{C}P^{l-1} \rightarrow \Sigma \bb{C}P^{l-2}/\bb{C}P^{l-3} = S^{2l-3}$, 
which belongs to the (2-local) third stable stem $\pi_3^s(S^0)\cong \bb{Z}/8$. 
More precisely, consider the following commutative diagram. 
Here both the rows and the middle two columns are part of cofiber sequences, and $\delta_{n-1}$ and $\delta_{n-2}$ are connecting maps in those cofiber sequences.

\begin{center}
\begin{tikzcd}[column sep = small]
 \bb{C}P^{l-1}/\bb{C}P^{l-2} \ar[r]  & \Sigma \bb{C}P^{l-2}/\bb{C}P^{l-3} \ar[r] & \Sigma \bb{C}P^{l-1}/\bb{C}P^{l-3} \ar[r] & \Sigma \bb{C}P^{l-1}/\bb{C}P^{l-2} \\
 \bb{C}P^{l-1}/\bb{C}P^{l-2} \ar[r] \ar[u, "="]  & \bb{C}P^{l}/\bb{C}P^{l-2} \ar[r] \ar[u, "\delta_{l-2}"] & \bb{C}P^{l}/\bb{C}P^{l-1} \ar[r] \ar[u, "\delta_{l-1}"] & \Sigma \bb{C}P^{l-1}/\bb{C}P^{l-2} \ar[u, "="] \\
\end{tikzcd}
\end{center}

When $l$ is even, $\bb{C}P^{l-1}/\bb{C}P^{l-3}$ splits as $S^{2l-2} \vee S^{2l-4}$. Denote by $p$ the quotient map 
$$ S^{2l-1} \vee S^{2l-3} = \Sigma \bb{C}P^{l-1}/\bb{C}P^{l-3} \rightarrow \Sigma \bb{C}P^{l-2}/\bb{C}P^{l-3} = S^{2l-3}, $$
which sections the inclusion $\Sigma \bb{C}P^{l-2}/\bb{C}P^{l-3} \rightarrow \Sigma \bb{C}P^{l-1}/\bb{C}P^{l-3}$. 
In this case $\lambda_l$ is the composite $p \circ \delta_{l-1}$.
When $l$ is odd, $\bb{C}P^{l}/\bb{C}P^{l-2}$ splits as $S^{2l} \vee S^{2l-2}$. Denote by $j$ the standard inclusion
$$S^{2l} = \bb{C}P^{l}/\bb{C}P^{l-1} \rightarrow \bb{C}P^{l}/\bb{C}P^{l-2} = S^{2l} \vee S^{2l-2},$$
which sections the quotient map $\bb{C}P^{l}/\bb{C}P^{l-2} \rightarrow \bb{C}P^{l}/\bb{C}P^{l-1}$. In this case $\lambda_l$ is the composite $\delta_{l-2} \circ j$.

In \cite{Mosher} Mosher determines the homotopy classes of the $\lambda_l$'s, which exhibit an 8-fold periodic behavior which we present as follows.
Write $\nu$ for the Hopf map $S^7\rightarrow S^4$ generating the 2-local third stable stem.

\begin{lemma}[Mosher \cite{Mosher}, Proposition 5.2] \label{AHSSd4}
Let $l\geq 4$ be an integer. The homotopy class of $\lambda_l$ satisfies an 8-fold periodicity as follows.
\begin{center}
\begin{tabular}{c | c c c c c c c c }
\hline
$l \mod 8$ & $0$ & $1$ & $2$ & $3$ & $4$ & $5$ & $6$ & $7$ \\
\hline
homotopy class of $\lambda_l$ & $\nu$ & $\nu$ & $0$ & $2\nu$ & $\nu$ & $\nu$ & $2\nu$ & $0$ \\
\hline
\end{tabular}
\end{center}
\end{lemma}

Proposition \ref{2local} follows immediately from Lemma \ref{AHSSd4}. When $l=2$ mod $8$, Lemma \ref{AHSSd4} suggests that the differential $d_4^{2l-4, -(2l-3)}$ we study is induced by the null map. It follows that the cokernel of $d_4^{2l-4, -(2l-3)}$ is $\bb{Z}/4$. 
For degree reasons there can be no further nontrivial differentials starting from or arriving at the diagonal $p+q = 0$. 
Thus we conclude that $\{\bb{C}P^l, \Sigma\bb{C}P^{\infty}_{l-2} \} \cong \bb{Z}/4$ when for $l=2$ mod 8.

In all other cases, determining $d_4^{2l-4, -(2l-3)}$ comes down to learning the homotopy class of $\lambda_l$, which by Lemma \ref{AHSSd4} is a multiple of $\nu$. 
Applying analysis of $\bb{C}P^l/\bb{C}P^{l-3}$ as above and using Mosher's Lemma gives Theorem \ref{2local}.

\subsubsection{Calculations at the prime 3}
In this subsection we prove Theorem \ref{3local}, and we work 3-locally throughout. At the prime 3, there is one possibly nonzero $d_4$-differential in the AHSS to be determined. This differential reflects the structure of $\bb{C}P^l/\bb{C}P^{l-3}$, and in some cases the bottom cell and the top cell in this stunted projective space are related by the 3-primary Steenrod operation $\mathcal{P}^1$, which detects the generator of the third stable stem at the prime 3. 

Our strategy for the 3-local calculation is exactly the same as in the  2-local case. 
First  we calculate some 3-local stable homotopy groups of $\Sigma\bb{C}P^{\infty}_{l-2}$. 
In the Appendix we shall prove the following 3-primary analogue of Lemma \ref{sthomotopygeneral}.

\begin{lemma} \label{sthomotopygeneral3}
The 3-primary stable homotopy groups $\pi_i^s(\Sigma\bb{C}P^{\infty}_n)$ for $i \leq 2n+4$ are as follows.
\begin{enumerate}
	\item $\pi_i^s(\Sigma\bb{C}P^{\infty}_n) = 0$ for $i\leq 2n$.
	\item $\pi_{2n+1}^s(\Sigma\bb{C}P^{\infty}_n) = \bb{Z}$, $\pi_{2n+2}^s(\Sigma\bb{C}P^{\infty}_n) = 0$, and $\pi_{2n+3}^s(\Sigma\bb{C}P^{\infty}_n) = \bb{Z}$.
	\item $\pi_{2n+4}^s(\Sigma\bb{C}P^{\infty}_n) = 0$ if $n = 1, 2 \mod 3$, and $\pi_{2n+4}^s(\Sigma\bb{C}P^{\infty}_n) = \bb{Z}/3$ if $n = 0 \mod 3$.
\end{enumerate}
\end{lemma}

We can now compute $\{\bb{C}P^l, \Sigma\bb{C}P^{\infty}_{l-2}\}$ via the AHSS which has
$$E_2^{p, q} = H^q\big{(}\bb{C}P^l; \pi_{-q}^s(\Sigma\bb{C}P^{\infty}_{l-2})\big{)} \Longrightarrow \{\bb{C}P^l, \Sigma^{p+q+1} \bb{C}P^{\infty}_{l-2}\}.$$

When $l = 0,1 \mod 3$, the $E_2$-page of the spectral sequence includes 

\begin{center}
\begin{tikzpicture}
  \matrix (m) [matrix of math nodes, nodes in empty cells, nodes={minimum width=5ex, minimum height=5ex, outer sep=-5pt}, column sep=1ex,row sep=1ex]{
        \quad\strut     &    &   \cdots  &  2l-4 & 2l-3  & 2l-2 & 2l-1 & 2l & \geq 2l+1  &  \strut  \\
				 \geq -(2l-4)   &    &  \cdots   &  0    & 0     & 0    & 0    & 0  & 0 & \\
          -(2l-3)       &    &    \cdots  & \bb{Z} & 0 & \bb{Z}   & 0    & \bb{Z} & 0 &  \\
          -(2l-2)       &    &    \cdots  & 0 & 0 & 0 & 0 & 0 & 0 & \\
					-(2l-1)       &    &   \cdots & \bb{Z} & 0 & \bb{Z} & 0 & \bb{Z} & 0 &  \\
					-2l           &    &  \cdots & 0 & 0 & 0 & 0 & 0 & 0 & \\
					};
					\draw[thick] (m-1-2.east) -- (m-6-2.east);
					\draw[thick] (m-1-1.south) -- (m-1-10.south);
\end{tikzpicture}.
\end{center}

Note that $E_2^{p, q}$ vanish for all $p>2l$ and for all $q>-(2l-3)$. In particular, all groups are zero along the diagonal $p+q = 0$. Thus $\{\bb{C}P^l, \Sigma\bb{C}P^{\infty}_{l-2}\} = 0$. 

When $l = 2 \mod 3$, the $E_2$-page of the spectral sequence includes

\begin{center}
\begin{tikzpicture}
  \matrix (m) [matrix of math nodes, nodes in empty cells, nodes={minimum width=5ex, minimum height=5ex, outer sep=-5pt}, column sep=1ex,row sep=1ex]{
        \quad\strut     &    &  \cdots  &  2l-4 & 2l-3  & 2l-2 & 2l-1 & 2l   & \geq 2l+1 & \strut  \\
				 \geq -(2l-4)   &    &  \cdots  &  0    & 0     & 0    & 0    & 0  & 0 & \\
          -(2l-3)       &    &  \cdots  & \bb{Z} & 0 & \bb{Z}   & 0    & \bb{Z} & 0 & \\
          -(2l-2)       &    &  \cdots  & 0 & 0 & 0 & 0 & 0 & 0 & \\
					-(2l-1)       &    &  \cdots & \bb{Z} & 0 & \bb{Z} & 0 & \bb{Z} & 0 & \\
					-2l           &    &  \cdots & \bb{Z}/3 & 0 & \bb{Z}/3 & 0 & \W & 0 & \\
					};
					\draw[dashed, ->] (m-5-6.south east) -- (m-6-8.north west);
					\draw[dashed, ->] (m-3-4.south east) -- (m-6-8.north west);
					\draw[thick] (m-1-2.east) -- (m-6-2.east);
					\draw[thick] (m-1-1.south) -- (m-1-10.south);
\end{tikzpicture}.
\end{center}

In this case the only nontrivial group along the diagonal $p+q = 0$ is the circled $E_2^{2l, -2l} \cong \bb{Z}/3$, and the only possible $d_2$-differential hitting this group is $d_2^{2l-2, -(2l-1)}: E_2^{2l-2, -(2l-1)}\rightarrow E_2^{2l, -2l}$. This differential is either induced by the null map or some suspension of $\eta: S^3\rightarrow S^2$, but since $\eta$ is 3-locally null the differential must vanish. So this copy of $\bb{Z}/3$ survives to the $E_4$-page.

There is a possibly nonzero differential $d_4^{2l-4, -(2l-3)}$ hitting $E_4^{2l, -2l} \cong \bb{Z}/3$. By the construction of the spectral sequence, this homomorphism $\bb{Z}\rightarrow \bb{Z}/3$ is induced by a map $S^{2l} = \bb{C}P^l/\bb{C}P^{l-1} \rightarrow \Sigma \bb{C}P^{l-2}/\bb{C}P^{l-3} = S^{2l-3}$, which belongs to the (3-local) third stable stem $\pi_3^s(S^0)\cong \bb{Z}/3$. 
We claim that this map must be null. Indeed, if it was essential then it must be detected by the 3-primary Steenrod operation $\mathcal{P}^1$, but $\mathcal{P}^1(x^{l-2}) = 0$ when $l = 2 \mod 3$.
So we conclude that the differential $d_4^{2l-4, -(2l-3)}$ is zero when $l=2 \mod 3$, and hence that $\{\bb{C}P^l, \Sigma\bb{C}P^{\infty}_{l-2}\} = \bb{Z}/3$ in this case.

We have proved that, 3-locally, $\{\bb{C}P^l, \Sigma\bb{C}P^{\infty}_{l-2}\}$ is zero if $l = 0, 1\mod 3$, and is isomorphic to $\bb{Z}/3$ when $l = 2 \mod 3$. Theorem \ref{3local} then follows.

\vspace{2mm}

\appendix

\section{The First Few Stable Homotopy Groups of \texorpdfstring{$\Sigma\bb{C}P^{\infty}_n$}{CPinfinityn}}

We prove Lemma \ref{sthomotopygeneral} and Lemma \ref{sthomotopygeneral3}, which compute, respectively, some first 2-local and 3-local stable homotopy groups of $\Sigma\bb{C}P^{\infty}_n$.

First note that $\Sigma\bb{C}P^{\infty}_n$ is $2n$-connected, and hence $\pi_i^s(\Sigma\bb{C}P^{\infty}_n) = 0$ for $i\leq 2n$. 
Secondly, $\pi_i^s(\Sigma\bb{C}P^{\infty}_n) = 0$ for $2n+1\leq i \leq 2n+4$ is controlled by the structure of the stunted projective space $\bb{C}P_n^{n+4} = \bb{C}P^{n+4}/\bb{C}P^{n-1}$.

\subsection{Proof of Lemma \ref{sthomotopygeneral}}
We start with the 2-primary calculations. The action of the mod 2 Steenrod algebra on the mod 2 cohomology of $\bb{C}P_n^{n+4}$ exhibits the following 8-fold periodic behavior. We present these actions in terms of diagrams as follows. 

\begin{center}
\begin{minipage}{0.2\textwidth}
  $n = 0 \mod 8$
\end{minipage}
\hspace{5mm}
\begin{minipage}{0.4\textwidth}
  \begin{tikzpicture}
			\draw [thin] (5,0) node[below]{{\footnotesize{$y_{2n+1}$}}}; \draw [thin] (6,0) node[below]{{\footnotesize{$y_{2n+3}$}}} -- (7,0) node[below]{{\footnotesize{$y_{2n+5}$}}}; \draw (8, 0) node[below]{{\footnotesize{$y_{2n+7}$}}} -- (9, 0) node[below]{{\footnotesize{$y_{2n+9}$}}}; 
		\draw [dashed] (9,0) arc (45:135:1.41cm); 
	\end{tikzpicture}
\end{minipage}
\end{center}

\begin{center}
\begin{minipage}{0.2\textwidth}
  $n = 1 \mod 8$
\end{minipage}
\hspace{5mm}
\begin{minipage}{0.4\textwidth}
  \begin{tikzpicture}		
			\draw (1,0) node[below]{{\footnotesize{$y_{2n+1}$}}} -- (2,0) node[below]{{\footnotesize{$y_{2n+3}$}}}; \draw [thin] (3,0) node[below]{{\footnotesize{$y_{2n+5}$}}} -- (4,0) node[below]{{\footnotesize{$y_{2n+7}$}}}; \draw [thin] (5,0) node[below]{{\footnotesize{$y_{2n+9}$}}};
	  \draw (4,0) arc (45:135:1.41cm) ; \draw (5,0) arc (-45:-135:1.41cm); 
	\end{tikzpicture}
\end{minipage}
\end{center}

\begin{center}
\begin{minipage}{0.2\textwidth}
  $n = 2 \mod 8$
\end{minipage}
\hspace{5mm}
\begin{minipage}{0.4\textwidth}
  \begin{tikzpicture}
	 \draw (0, 0) node[below]{{\footnotesize{$y_{2n+1}$}}};
			\draw (1,0) node[below]{{\footnotesize{$y_{2n+3}$}}} -- (2,0) node[below]{{\footnotesize{$y_{2n+5}$}}}; \draw [thin] (3,0) node[below]{{\footnotesize{$y_{2n+7}$}}} -- (4,0) node[below]{{\footnotesize{$y_{2n+9}$}}};
	  \draw (2,0) arc (45:135:1.41cm); \draw (3,0) arc (-45:-135:1.41cm); 
	\end{tikzpicture}
\end{minipage}
\end{center}

\begin{center}
\begin{minipage}{0.2\textwidth}
  $n = 3 \mod 8$
\end{minipage}
\hspace{5mm}
\begin{minipage}{0.4\textwidth}
  \begin{tikzpicture}
			\draw (1,0) node[below]{{\footnotesize{$y_{2n+1}$}}} -- (2,0) node[below]{{\footnotesize{$y_{2n+3}$}}}; \draw [thin] (3,0) node[below]{{\footnotesize{$y_{2n+5}$}}} -- (4,0) node[below]{{\footnotesize{$y_{2n+7}$}}}; \draw (5,0) node[below]{{\footnotesize{$y_{2n+9}$}}};
	  \draw (3,0) arc (-45:-135:1.41cm); 
	\end{tikzpicture}
\end{minipage}
\end{center}

\begin{center}
\begin{minipage}{0.2\textwidth}
  $n = 4 \mod 8$
\end{minipage}
\hspace{5mm}
\begin{minipage}{0.4\textwidth}
  \begin{tikzpicture}
			\draw (2,0) node[below]{{\footnotesize{$y_{2n+1}$}}}; \draw [thin] (3,0) node[below]{{\footnotesize{$y_{2n+3}$}}} -- (4,0) node[below]{{\footnotesize{$y_{2n+5}$}}}; \draw [thin] (5,0) node[below]{{\footnotesize{$y_{2n+7}$}}} -- (6,0) node[below]{{\footnotesize{$y_{2n+9}$}}};
    \draw (6,0) arc (45:135:1.41cm); 
		\draw [dashed] (6,0) arc (45:135:2.83cm);
	\end{tikzpicture}
\end{minipage}
\end{center}

\begin{center}
\begin{minipage}{0.2\textwidth}
  $n = 5 \mod 8$
\end{minipage}
\hspace{5mm}
\begin{minipage}{0.4\textwidth}
  \begin{tikzpicture}
			\draw [thin] (3,0) node[below]{{\footnotesize{$y_{2n+1}$}}} -- (4,0) node[below]{{\footnotesize{$y_{2n+3}$}}}; \draw [thin] (5,0) node[below]{{\footnotesize{$y_{2n+5}$}}} -- (6,0) node[below]{{\footnotesize{$y_{2n+7}$}}};
   \draw (6,0) arc (45:135:1.41cm); 
		\draw (7, 0) node[below]{{\footnotesize{$y_{2n+9}$}}}; \draw (7,0) arc (-45:-135:1.41cm); \draw [dashed] (7,0) arc (-45:-135:2.83cm); 
	\end{tikzpicture}
\end{minipage}
\end{center}

\begin{center}
\begin{minipage}{0.2\textwidth}
  $n = 6 \mod 8$
\end{minipage}
\hspace{5mm}
\begin{minipage}{0.4\textwidth}
  \begin{tikzpicture}
			\draw [thin] (5,0) node[below]{{\footnotesize{$y_{2n+1}$}}}; \draw [thin] (6,0) node[below]{{\footnotesize{$y_{2n+3}$}}} -- (7,0) node[below]{{\footnotesize{$y_{2n+5}$}}}; \draw (8, 0) node[below]{{\footnotesize{$y_{2n+7}$}}} -- (9, 0) node[below]{{\footnotesize{$y_{2n+9}$}}}; 
   \draw (7,0) arc (45:135:1.41cm); \draw (8,0) arc (-45:-135:1.41cm); 
		\draw [dashed] (9,0) arc (45:135:2.83cm);
	\end{tikzpicture}
\end{minipage}
\end{center}

\begin{center}
\begin{minipage}{0.2\textwidth}
  $n = 7 \mod 8$
\end{minipage}
\hspace{5mm}
\begin{minipage}{0.4\textwidth}
  \begin{tikzpicture}
			\draw [thin] (5,0) node[below]{{\footnotesize{$y_{2n+1}$}}} -- (6,0) node[below]{{\footnotesize{$y_{2n+3}$}}}; \draw (7,0) node[below]{{\footnotesize{$y_{2n+5}$}}} --  (8, 0) node[below]{{\footnotesize{$y_{2n+7}$}}}; \draw (9, 0) node[below]{{\footnotesize{$y_{2n+9}$}}}; 
   \draw (7,0) arc (-45:-135:1.41cm); 
		\draw [dashed] (9,0) arc (-45:-135:2.83cm); 
	\end{tikzpicture}
\end{minipage}
\end{center}

We have constructed  explicit minimal $\cal{A}$-resolution s
$$\cdots \rightarrow P_s \stackrel{\partial_s}{\longrightarrow} \cdots P_2 \stackrel{\partial_2}{\longrightarrow} P_1 \stackrel{\partial_1}{\longrightarrow} P_0 \stackrel{\epsilon}{\longrightarrow} H^*(\bb{C}P_n^{n+4}) \rightarrow 0$$
in each case to compute  Adams $E_2$ pages. 
We present details in only one example, namely the case $n=2$, with other cases being similar.  In all cases,
our hand calculations were kindly verified by Robert Bruner using his computer code \cite{Bruner}.

Focusing on the case of $\Sigma\bb{C}P^{\infty}_2$, which is needed to prove Lemma \ref{sthomotopy}, 
we recall that $H^*(\Sigma\bb{C}P^{\infty}_2)$ has the following behavior.

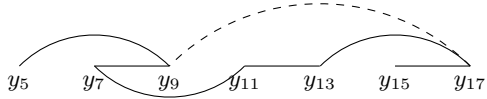
\begin{center}
  \begin{tikzpicture}
	 \draw (0, 0) node[below]{{\footnotesize{$y_5$}}};
			\draw (1,0) node[below]{{\footnotesize{$y_7$}}} -- (2,0) node[below]{{\footnotesize{$y_9$}}}; \draw [thin] (3,0) node[below]{{\footnotesize{$y_{11}$}}} -- (4,0) node[below]{{\footnotesize{$y_{13}$}}}; \draw [thin] (5,0) node[below]{{\footnotesize{$y_{15}$}}} -- (6,0) node[below]{{\footnotesize{$y_{17}$}}};
	  \draw (2,0) arc (45:135:1.41cm); \draw (3,0) arc (-45:-135:1.41cm); \draw (6,0) arc (45:135:1.41cm);
		\draw [dashed] (6,0) arc (45:135:2.83cm);
	\end{tikzpicture}
\end{center}

\underline{Filtration $s=0$}. To define $P_0$ which surjects onto $H^*(\Sigma\bb{C}P^{\infty}_2)$,
we introduce a free generator $e_{0, 5}$ in degree 5 to kill $y_5$, and a free generator $e_{0, 7}$ in degree 7 to kill 
$y_7$. That is, 
$$P_0:= \cal{A}e_{0, 5} \oplus \cal{A}e_{0, 7}\oplus \cdots, \quad \epsilon(e_{0, 5}) = y_5, \epsilon(e_{0, 7}) = y_7, \cdots$$
The next free generator to introduce would be $e_{0, 15}$ in degree 15 so that $\epsilon(e_{0, 15}) = y_{15}$,
so $P_0$ has no summands generated in degree $t$ for $7<t<15$.

A basis for $\Ker(\epsilon)$ in degrees $t\leq 13$ is presented below.

\vspace{1mm}

\begin{center}
\begin{tabular}{l|l l l l l l}
{\rm Deg} 6 & \quad $\Sq^1 e_{0, 5}$ & & & & & \\
{\rm Deg} 7 & \quad $\Sq^2 e_{0, 5}$ & & & & &  \\
{\rm Deg} 8 & \quad $\Sq^3 e_{0, 5}$ &  $\Sq^2\Sq^1 e_{0, 5}$ & \quad $\Sq^1 e_{0, 7}$ & & & \\
{\rm Deg} 9 &  &  $\Sq^3\Sq^1 e_{0, 5}$ & \quad $\Sq^4 e_{0, 5} + \Sq^2 e_{0, 7}$ & & & \\
{\rm Deg} 10 & \quad $\Sq^5 e_{0, 5}$ &  $\Sq^4\Sq^1 e_{0, 5}$ & \quad $\Sq^3 e_{0, 7}$ &  \quad $\Sq^2\Sq^1 e_{0, 7}$ & & \\
{\rm Deg} 11 & \quad $\Sq^6 e_{0, 5}$ & $\Sq^5\Sq^1 e_{0, 5}$ & \quad $\Sq^4\Sq^2 e_{0, 5}$ & \quad $\Sq^3\Sq^1 e_{0, 7}$ & & \\
{\rm Deg} 12 & \quad $\Sq^7 e_{0, 5}$ & $\Sq^6\Sq^1 e_{0, 5}$ & \quad $\Sq^5\Sq^2 e_{0, 5}$ & \quad $\Sq^4\Sq^2\Sq^1 e_{0, 5}$ & $\Sq^5 e_{0, 7}$ & $\Sq^4\Sq^1 e_{0, 7}$ \\
{\rm Deg} 13 & \quad $\Sq^8 e_{0, 5}$ & $\Sq^7\Sq^1 e_{0, 5}$ & \quad $\Sq^6\Sq^2 e_{0, 5}$ & \quad $\Sq^4\Sq^2\Sq^1 e_{0, 5}$ & $\Sq^5\Sq^1 e_{0, 7}$ & $\Sq^4\Sq^2 e_{0, 7}$ \\
\end{tabular}
\end{center}

\vspace{1mm}

\underline{Filtration $s=1$}. With the kernel of the surjection $\epsilon$ 
from $P_0$ to the cohomology of $\Sigma\bb{C}P^{\infty}_2$
in hand, we construct $P_1$ together with a surjection onto $\Ker(\epsilon)$. We define
$$P_1:= \cal{A}e_{1, 6} \oplus \cal{A}e_{1, 7}\oplus \cal{A}e_{1, 8} \oplus \cal{A}e_{1, 9} \oplus \cdots,$$
$$\partial_1 e_{1, 6} = \Sq^1 e_{0, 5}, \partial_1 e_{1, 7} = \Sq^2 e_{0, 5}, \partial_1 e_{1, 8} = \Sq^1 e_{0, 7}, \partial_1 e_{1, 9} = \Sq^4 e_{0, 5} + \Sq^2 e_{0, 7}, \cdots.$$

Note that the first element of $\Ker(\epsilon)$ that is not yet in $\partial_1(\cal{A}e_{1, 6} \oplus \cal{A}e_{1, 7}\oplus \cal{A}e_{1, 8} \oplus \cal{A}e_{1, 9})$ is $\Sq^8 e_{0, 5}$ in degree 13. So the next free generator to introduce to $P_1$ is $e_{1, 13}$ so that $\partial_1 e_{1, 13} = \Sq^8 e_{0, 5}$. 
In particular, we see that $P_1$ has no components of degree $t$ for $9<t<13$.
Furthermore, we note that $e_{1, 6}$ is connected with $e_{0, 5}$ by $h_0$, $e_{1, 7}$ is connected with $e_{0, 5}$ by $h_1$, $e_{1, 7}$ is connected with $e_{0, 7}$ by $h_0$, $e_{1, 9}$ is connected with $e_{0, 5}$ by $h_2$ and with $e_{0, 7}$ by $h_1$.

A basis for $\Ker(\partial_1)$ in degrees $t\leq 13$ is presented below.

\vspace{1mm}

\begin{center}
\begin{tabular}{l|l l l l l l}
{\rm Deg} 7 & \quad $\Sq^1 e_{1, 6}$ & & & & & \\
{\rm Deg} 8 & \quad   & & & & &  \\
{\rm Deg} 9 & \quad $\Sq^2\Sq^1 e_{1, 6}$ & $\Sq^3 e_{1, 6} + \Sq^2 e_{1, 7}$ & $\Sq^1 e_{1, 8}$ & & & \\
{\rm Deg} 10 & \quad $\Sq^3\Sq^1 e_{1, 6}$ & $\Sq^3 e_{1, 7}$ &  & & & \\
{\rm Deg} 11 & \quad $\Sq^4\Sq^1 e_{1, 6}$ & $\Sq^5 e_{1, 6} + \Sq^3\Sq^1 e_{1, 7}$ &  $\Sq^2\Sq^1 e_{1, 8}$ & & \\
{\rm Deg} 12 & \quad $\Sq^5\Sq^1 e_{1, 6}$ & $\Sq^5 e_{1, 7} + \Sq^4\Sq^1 e_{1, 7}$ & $\Sq^3\Sq^1 e_{1, 8}$ & & \\
{\rm Deg} 13 & \quad $\Sq^6\Sq^1 e_{1, 6}$ & $\Sq^5\Sq^2 e_{1, 6} + \Sq^4\Sq^2 e_{1, 7}$ & $\Sq^4\Sq^1 e_{1, 8}$ & $\Sq^4\Sq^2\Sq^1 e_{1, 6}$ & $\Sq^5\Sq^1 e_{1, 7}$ &   \\
 &  & $\Sq^7e_{1, 6}+ \Sq^5 e_{1, 8} + \Sq^3\Sq^1 e_{1, 9}$  & & & &
\end{tabular}
\end{center}

\vspace{1mm}

\underline{Filtration $s=2$}. We construct $P_2$ together with a surjection onto $\Ker(\partial_1)$. We define
$$P_2:= \cal{A}e_{2, 7} \oplus \cal{A}e_{2, 9}\oplus \cal{A}e_{2, 9}' \oplus \cdots,$$
$$\partial_2 e_{2, 7} = \Sq^1 e_{1, 6}, \partial_2 e_{2, 9} = \Sq^1 e_{1, 8}, \partial_2 e_{2, 9}' = \Sq^3 e_{1, 6} + \Sq^2 e_{1, 7}, \cdots.$$

Note that the first element of $\Ker(\epsilon)$ that is not yet in $\partial_2(\cal{A}e_{2, 7} \oplus \cal{A}e_{2, 9}\oplus \cal{A}e_{2, 9}')$ is $\Sq^7e_{1, 6}+ \Sq^5 e_{1, 8} + \Sq^3\Sq^1 e_{1, 9}$ in degree 13. So the next free generator to introduce to $P_2$ is $e_{2, 13}$ so that $\partial_2 e_{2, 13} = \Sq^7e_{1, 6}+ \Sq^5 e_{1, 8} + \Sq^3\Sq^1 e_{1, 9}$. In particular, we see that $P_2$ has no summands of degree $t$ for $9<t<13$. 
Furthermore,  $e_{2, 7}$ is connected with $e_{1, 6}$ by $h_0$, $e_{2, 9}$ is connected with $e_{1, 8}$ by $h_0$, and $e_{2, 9}'$ is connected with $e_{1, 7}$ by $h_1$.

A basis for $\Ker(\partial_2)$ in degrees $t\leq 13$ is presented below.

\vspace{1mm}

\begin{center}
\begin{tabular}{l| l l l}
{\rm Deg} 8 & \quad  $\Sq^1 e_{2, 7}$ & &  \\
{\rm Deg} 9 & \quad   & &   \\
{\rm Deg} 10 & \quad $\Sq^2\Sq^1 e_{2, 7}$ & \quad $\Sq^1 e_{2, 9}$ &  \\
{\rm Deg} 11 & \quad $\Sq^3\Sq^1 e_{2, 7}$ &  & \\
{\rm Deg} 12 & \quad $\Sq^4\Sq^1 e_{2, 7}$ & \quad $\Sq^2\Sq^1 e_{2, 9}$ & \quad $\Sq^5 e_{2, 7} + \Sq^3 e_{2, 9}$ \\
{\rm Deg} 13 & \quad $\Sq^5\Sq^1 e_{2, 7}$ & \quad $\Sq^3\Sq^1 e_{2, 9}$ & 
\end{tabular}
\end{center}

\vspace{1mm}

\underline{Filtration $s=3$}. We construct $P_3$ together with a surjection onto $\Ker(\partial_2)$. We define
$$P_3:= \cal{A}e_{3, 8} \oplus \cal{A}e_{3, 10}\oplus \cal{A}e_{3, 12} \oplus \cdots,$$
$$\partial_3 e_{3, 8} = \Sq^1 e_{2, 7}, \partial_3 e_{3, 10} = \Sq^1 e_{2, 9}, \partial_3 e_{3, 12} = \Sq^5 e_{2, 7} + \Sq^3 e_{2, 9}, \cdots.$$

Every element of $\Ker(\partial_2)$ in the range given above is contained in $\partial_3(\cal{A}e_{3, 8} \oplus \cal{A}e_{3, 10}\oplus \cal{A}e_{3, 12})$. 
Furthermore,  $e_{3, 8}$ is connected with $e_{2, 7}$ by $h_0$, and $e_{3, 10}$ is connected with $e_{2, 9}$ by $h_0$.

A basis for $\Ker(\partial_3)$ in degrees $t\leq 13$ is presented below.

\vspace{1mm}

\begin{center}
\begin{tabular}{l| l l}
{\rm Deg} 9 & \quad $\Sq^1 e_{3, 8}$ & \\
 {\rm Deg} 10 & \quad  & \\
{\rm Deg} 11 & \quad $\Sq^2\Sq^1 e_{3, 8}$ & \quad $\Sq^1 e_{3, 10}$ \\
{\rm Deg} 12 & \quad $\Sq^3\Sq^1 e_{3, 8}$ &  \\
{\rm Deg} 13 & \quad $\Sq^4\Sq^1 e_{3, 8}$ & \quad $\Sq^2\Sq^1 e_{3, 10}$
\end{tabular}
\end{center}

\vspace{1mm}

\underline{Filtration $s=4$}. We construct $P_4$ together with a surjection onto $\Ker(\partial_3)$. We define
$$P_3:= \cal{A}e_{4, 9} \oplus \cal{A}e_{4, 11}\oplus \cdots,$$
$$\partial_4 e_{4, 9} = \Sq^1 e_{3, 8}, \partial_4 e_{4, 11} = \Sq^1 e_{3, 10}, \cdots.$$

Every element of $\Ker(\partial_3)$ in the above range is contained in $\partial_4(\cal{A}e_{4, 9} \oplus \cal{A}e_{4, 11})$.
Furthermore,  $e_{4, 9}$ is connected with $e_{3, 8}$ by $h_0$, and $e_{4, 11}$ is connected with $e_{3, 10}$ by $h_0$.

\underline{Filtration $s>4$}.
Inductively, suppose that the basis for $\Ker(\partial_s)$ ($s\geq 4$) in degrees $t\leq s+8$ can be taken as follows.

\vspace{1mm}

\begin{center}
\begin{tabular}{l| l l}
{\rm Deg} \; s+6 & \quad $\Sq^1 e_{s, s+5}$ & \\
{\rm Deg} \; s+7 & \quad  & \\
{\rm Deg} \; s+8 & \quad $\Sq^2\Sq^1 e_{s, s+5}$ & \quad $\Sq^1 e_{s, s+7}$ \\
{\rm Deg} \; s+9 & \quad $\Sq^3\Sq^1 e_{s, s+5}$ &  \\
{\rm Deg} \; s+10 & \quad $\Sq^4\Sq^1 e_{s, s+5}$ & \quad $\Sq^2\Sq^1 e_{s, s+7}$
\end{tabular}
\end{center}

\vspace{1mm}

Then to build $P_{s+1}$ together with a surjection onto $\Ker(\partial_s)$, we let
$$P_{s+1}:= \cal{A} e_{s+1, s+6} \oplus \cal{A} e_{s+1, s+8} \oplus \cdots,$$
$$\partial_{s+1} e_{s+1, s+6} = \Sq^1 e_{s, s+5}, \partial_{s+1} e_{s+1, s+8} = \Sq^1 e_{s, s+7}.$$

Thus every element of $\Ker(\partial_s)$ in the above range is contained in $\partial_{s+1}(\cal{A}e_{s+1, s+6} \oplus \cal{A}e_{s+1, s+8})$.
Furthermore, $e_{s+1, s+6}$ is connected with $e_{s, s+5}$ by $h_0$, and $e_{s+1, s+8}$ is connected with $e_{s, s+7}$ by $h_0$.
It also follows that a basis for $\Ker(\partial_{s+1})$ in degrees $t\leq s+9$ can be taken as follows.

\vspace{1mm}

\begin{center}
\begin{tabular}{l| l l}
{\rm Deg} \; s+7 & \quad $\Sq^1 e_{s+1, s+6}$ & \\
{\rm Deg} \; s+8 & \quad   & \\
{\rm Deg} \; s+9 & \quad $\Sq^2\Sq^1 e_{s+1, s+6}$ & \quad $\Sq^1 e_{s+1, s+8}$ \\
{\rm Deg} \; s+10 & \quad $\Sq^3\Sq^1 e_{s+1, s+6}$ &  \\
{\rm Deg} \; s+11 & \quad $\Sq^4\Sq^1 e_{s+1, s+6}$ & \quad $\Sq^2\Sq^1 e_{s+1, s+8}$ 
\end{tabular}
\end{center}

\vspace{1mm}

Thus we have established the following.

\begin{proposition}
The Adams $E_2$-page for $\Sigma\bb{C}P^{\infty}_2$ satisfies:

\begin{enumerate}
	\item  The groups vanish whenever $t-s \leq 4$.
	\item There is a single $h_0$-tower starting at $(s, t-s) = (0, 5)$, and there is nothing else along $t-s = 5$.
	\item There is a single $\bb{Z}_2$ at $(s, t-s) = (1, 6)$, and there is nothing else along $t-s = 6$.
	\item There is a single $h_0$-tower starting at $(s, t-s) = (0, 7)$, a single $\bb{Z}_2$ at $(s, t-s) = (2, 7)$, and there is nothing else along $t-s = 7$.
	\item There is a single $\bb{Z}_2$ at $(s, t-s) = (1, 8)$, and there is nothing else along $t-s = 8$.
\end{enumerate}
\end{proposition}

\begin{center}
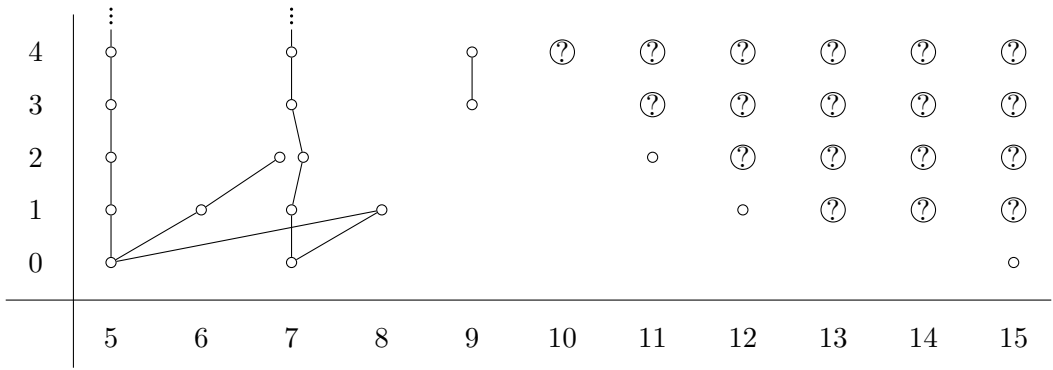
\printpage[name =  first results, title = The Adams $E_2$-page of $\Sigma\bb{C}P^{\infty}_2$]
\end{center}

Any differential in the range $t-s\leq 8$ must be trivial, so this is also the Adams $E_\infty$ and thus stable
homotopy in this range.  
More generally if $n=2 \mod 8$, this argument shows that $\Sigma\bb{C}P^{\infty}_n$ has the following stable homotopy groups:
$$\pi^s_{2n+1}(\Sigma\bb{C}P^{\infty}_n) = \bb{Z}, \quad \pi^s_{2n+2}(\Sigma\bb{C}P^{\infty}_n) = \bb{Z}/2, \quad \pi^s_{2n+3}(\Sigma\bb{C}P^{\infty}_n) = \bb{Z}\oplus \bb{Z}/2, \quad \pi^s_{2n+4}(\Sigma\bb{C}P^{\infty}_n) = \bb{Z}/2.$$

The other seven cases can be proved in the exact same way. We obtain the following conclusions.

When $n=0, 4$ mod 8, the Adams $E_2$-page for $\Sigma\bb{C}P^{\infty}_n$ begins as follows.

\begin{center}
\DeclareSseqGroup\tower {} {
    \class(0,0)
    \foreach \y in {1,...,9} {
        \class(0,\y)
\structline
}}

\begin{sseqdata}[name = third result, Adams grading, yscale = 0.7, xscale = 1.5, y range = {0}{4}, x range = {4}{10}, no x ticks]
\begin{scope}[ background ]
\node at (5,\ymin - 1.5) {2n+1};
\node at (6,\ymin - 1.5) {2n+2};
\node at (7,\ymin - 1.5) {2n+3};
\node at (8,\ymin - 1.5) {2n+4};
\node at (9,\ymin - 1.5) {2n+5};
\end{scope}
\tower(5, 0)
\class(6, 1)
\class(7, 2)
\tower(7, 0)
\class(8, 1) \class(8, 2) \class(8, 3)
\structline(5, 0)(6, 1)
\structline(6, 1)(7, 2)
\structline(7, 2)(8, 3)
\structline(8, 1)(8, 2)
\structline(8, 2)(8, 3)
\structline(5, 0)(8, 1)
\structline(5, 1)(8, 2)
\structline(5, 2)(8, 3)
\tower(9, 1)
\d[red, dashed]2(9, 1)
\end{sseqdata}
\printpage[name = third result, scale = 0.9]
\end{center}

When $n=1, 5$ mod 8, the Adams $E_2$-page for $\Sigma\bb{C}P^{\infty}_n$ begins as follows.

\begin{center}
\DeclareSseqGroup\tower {} {
    \class(0,0)
    \foreach \y in {1,...,9} {
        \class(0,\y)
\structline
}}

\begin{sseqdata}[name = fourth result, Adams grading, yscale = 0.7, xscale = 1.5, y range = {0}{4}, x range = {4}{10}, no x ticks]
\begin{scope}[ background ]
\node at (5,\ymin - 1.5) {2n+1};
\node at (6,\ymin - 1.5) {2n+2};
\node at (7,\ymin - 1.5) {2n+3};
\node at (8,\ymin - 1.5) {2n+4};
\node at (9,\ymin - 1.5) {2n+5};
\end{scope}
\tower(5, 0)
\tower(7, 1)
\class(8, 1)
\class(8, 2)
\tower(9, 0)
\structline(8, 1)(8, 2)
\structline(5, 0)(8, 1)
\structline(5, 1)(8, 2)
\d[red, dashed]2(9, 0)
\end{sseqdata}
\printpage[name = fourth result, scale = 0.9]
\end{center}

When $n=2, 6$ mod 8, the Adams $E_2$-page for $\Sigma\bb{C}P^{\infty}_n$ begins as follows.

\begin{center}
\DeclareSseqGroup\tower {} {
    \class(0,0)
    \foreach \y in {1,...,9} {
        \class(0,\y)
\structline
}}

\begin{sseqdata}[name = first result, Adams grading, yscale = 0.7, xscale = 1.5, y range = {0}{4}, x range = {4}{10}, no x ticks]
\begin{scope}[ background ]
\node at (5,\ymin - 1.5) {2n+1};
\node at (6,\ymin - 1.5) {2n+2};
\node at (7,\ymin - 1.5) {2n+3};
\node at (8,\ymin - 1.5) {2n+4};
\node at (9,\ymin - 1.5) {2n+5};
\end{scope}
\tower(5, 0)
\class(6, 1)
\class(7, 2)
\tower(7, 0)
\class(8, 1)
\structline(5, 0)(6, 1)
\structline(6, 1)(7, 2)
\structline(7, 0)(8, 1)
\structline(5, 0)(8, 1)
\class(9, 3) 
\class(9, 4)
\structline(9, 3)(9, 4)
\end{sseqdata}
\printpage[name = first result, scale = 0.9]
\end{center}

When $n=3, 7$ mod 8, the Adams $E_2$-page for $\Sigma\bb{C}P^{\infty}_n$ begins as follows.

\begin{center}
\DeclareSseqGroup\tower {} {
    \class(0,0)
    \foreach \y in {1,...,9} {
        \class(0,\y)
\structline
}}

\begin{sseqdata}[name = second result, Adams grading, yscale = 0.7, xscale = 1.5, y range = {0}{4}, x range = {4}{10}, no x ticks]
\begin{scope}[ background ]
\node at (5,\ymin - 1.5) {2n+1};
\node at (6,\ymin - 1.5) {2n+2};
\node at (7,\ymin - 1.5) {2n+3};
\node at (8,\ymin - 1.5) {2n+4};
\node at (9,\ymin - 1.5) {2n+5};
\end{scope}
\tower(5, 0)
\tower(7, 0)
\tower(9, 2)
\end{sseqdata}
\printpage[name = second result, scale = 0.9]
\end{center}

In the last two Adams charts there is no room for any nontrivial differential $d_r$ affecting the region $t-s \leq 2n+4$, for all $r\geq 2$. 
The following can therefore be concluded immediately.

\begin{itemize}
	\item When $n=2, 6 \mod 8$, 
	$$\pi^s_{2n+1}(\Sigma\bb{C}P^{\infty}_n) = \bb{Z}, \quad \pi^s_{2n+2}(\Sigma\bb{C}P^{\infty}_n) = \bb{Z}/2, \quad \pi^s_{2n+3}(\Sigma\bb{C}P^{\infty}_n) = \bb{Z}\oplus \bb{Z}/2, \quad \pi^s_{2n+4}(\Sigma\bb{C}P^{\infty}_n) = \bb{Z}/2.$$
	\item When $n=3, 7 \mod 8$, 
	$$\pi^s_{2n+1}(\Sigma\bb{C}P^{\infty}_n) = \bb{Z}, \quad \pi^s_{2n+2}(\Sigma\bb{C}P^{\infty}_n) = 0, \quad \pi^s_{2n+3}(\Sigma\bb{C}P^{\infty}_n) = \bb{Z}, \quad \pi^s_{2n+4}(\Sigma\bb{C}P^{\infty}_n) = 0.$$
\end{itemize}

However, in each of the first two Adams charts (i.e., when $n=0, 4$ or $1, 5 \mod 8$) there is a possible $d_2$ differential (which is presented as the red dashed arrow in the chart), which has to do with the determination of $\pi_{2n+4}^s(\Sigma\bb{C}P^{\infty}_n)$. To determine these differentials, we recall some classical results. 

First a result of Toda \cite{Toda}  relates the stable homotopy groups of $\bb{C}P^{\infty}_n$ to the metastable homotopy groups of unitary groups.

\begin{theorem}[Toda \cite{Toda}] \label{Toda}
Let $0\leq t <n$. Then $\pi^s_{2n+2t+1}(\bb{C}P^{\infty}_n) = \pi_{2n+2t+1} U(n)$.
\end{theorem}

Secondly, the relevant  homotopy groups were computed by Matsunaga \cite{Matsunaga}.

\begin{theorem}[Matsunaga \cite{Matsunaga}] \label{Matsunaga}
Two-locally,  metastable homotopy groups of $U(n)$ are given as follows
\begin{enumerate}
	\item $\pi_{2n+3} U(n) = \bb{Z}/8$ when $n = 0 \mod 8$.
	\item $\pi_{2n+3} U(n) = \bb{Z}/4$ when $n = 4 \mod 8$.
	\item $\pi_{2n+3} U(n) = \bb{Z}/4$ when $n = 1 \mod 8$.
	\item $\pi_{2n+3} U(n) = \bb{Z}/2$ when $n = 5 \mod 8$.
\end{enumerate}
\end{theorem}

So by Theorem \ref{Toda}, $\pi_{2n+4}^s(\Sigma\bb{C}P^{\infty}_n)$ is given by the list of \ref{Matsunaga}. 
It follows that the $d_2$ differential of interest must be zero when $n = 0, 1 \mod 8$, 
and must be an isomorphism when $n = 4, 5 \mod 8$.
We can now conclude the followings.

\begin{itemize}
	\item When $n=0 \mod 8$, 
	$$\pi^s_{2n+1}(\Sigma\bb{C}P^{\infty}_n) = \bb{Z}, \quad \pi^s_{2n+2}(\Sigma\bb{C}P^{\infty}_n) = \bb{Z}/2, \quad \pi^s_{2n+3}(\Sigma\bb{C}P^{\infty}_n) = \bb{Z}\oplus \bb{Z}/2, \quad \pi^s_{2n+4}(\Sigma\bb{C}P^{\infty}_n) = \bb{Z}/8.$$
	\item When $n=4 \mod 8$, 
	$$\pi^s_{2n+1}(\Sigma\bb{C}P^{\infty}_n) = \bb{Z}, \quad \pi^s_{2n+2}(\Sigma\bb{C}P^{\infty}_n) = \bb{Z}/2, \quad \pi^s_{2n+3}(\Sigma\bb{C}P^{\infty}_n) = \bb{Z}\oplus \bb{Z}/2, \quad \pi^s_{2n+4}(\Sigma\bb{C}P^{\infty}_n) = \bb{Z}/4.$$
	\item When $n=1 \mod 8$, 
	$$\pi^s_{2n+1}(\Sigma\bb{C}P^{\infty}_n) = \bb{Z}, \quad \pi^s_{2n+2}(\Sigma\bb{C}P^{\infty}_n) = 0, \quad \pi^s_{2n+3}(\Sigma\bb{C}P^{\infty}_n) = \bb{Z}, \quad \pi^s_{2n+4}(\Sigma\bb{C}P^{\infty}_n) = \bb{Z}/4.$$
	\item When $n=5 \mod 8$, 
	$$\pi^s_{2n+1}(\Sigma\bb{C}P^{\infty}_n) = \bb{Z}, \quad \pi^s_{2n+2}(\Sigma\bb{C}P^{\infty}_n) = 0, \quad \pi^s_{2n+3}(\Sigma\bb{C}P^{\infty}_n) = \bb{Z}, \quad \pi^s_{2n+4}(\Sigma\bb{C}P^{\infty}_n) = \bb{Z}/2.$$
\end{itemize}

This completes the 2-local computation and proves Lemma \ref{sthomotopygeneral}.

\subsection{Proof of Lemma \ref{sthomotopygeneral3}}
We now work 3-locally to prove Lemma \ref{sthomotopygeneral3}. The strategy is similar. 
The action of the mod 3 Steenrod algebra on the mod 3 cohomology of $\bb{C}P_n^{n+4}$ exhibits the following 3-fold periodicity. 
We present these actions in terms of diagrams. They correspond, respectively, to cases $n= 0, 1, 2$ mod 3.
Here each curved segment indicates a nontrivial action of $\cal{P}^1$. 

\begin{center}
\begin{minipage}{0.2\textwidth}
  $n = 0 \mod 3$
\end{minipage}
\hspace{5mm}
\begin{minipage}{0.4\textwidth}
  \begin{tikzpicture}
	 \draw (0, 0) node[below]{{\footnotesize{$y_{2n+1}$}}}; \draw (1,0) node[below]{{\footnotesize{$y_{2n+3}$}}}; \draw (2,0) node[below]{{\footnotesize{$y_{2n+5}$}}}; \draw (3,0) node[below]{{\footnotesize{$y_{2n+7}$}}}; \draw (4,0) node[below]{{\footnotesize{$y_{2n+9}$}}};
	  \draw (3, 0) arc (45:135:1.41cm); \draw (4, 0) arc (-45:-135:1.41cm); 
	\end{tikzpicture}
\end{minipage}
\end{center}

\begin{center}
\begin{minipage}{0.2\textwidth}
  $n = 1 \mod 3$
\end{minipage}
\hspace{5mm}
\begin{minipage}{0.4\textwidth}
  \begin{tikzpicture}
	 \draw (0, 0) node[below]{{\footnotesize{$y_{2n+1}$}}}; \draw (1,0) node[below]{{\footnotesize{$y_{2n+3}$}}}; \draw (2,0) node[below]{{\footnotesize{$y_{2n+5}$}}}; \draw (3,0) node[below]{{\footnotesize{$y_{2n+7}$}}}; \draw (4,0) node[below]{{\footnotesize{$y_{2n+9}$}}};
	  \draw (2, 0) arc (45:135:1.41cm); \draw (3, 0) arc (-45:-135:1.41cm);
	\end{tikzpicture}
\end{minipage}
\end{center}

\begin{center}
\begin{minipage}{0.2\textwidth}
  $n = 2 \mod 3$
\end{minipage}
\hspace{5mm}
\begin{minipage}{0.4\textwidth}
  \begin{tikzpicture}
	 \draw (0, 0) node[below]{{\footnotesize{$y_{2n+1}$}}}; \draw (1,0) node[below]{{\footnotesize{$y_{2n+3}$}}}; \draw (2,0) node[below]{{\footnotesize{$y_{2n+5}$}}}; \draw (3,0) node[below]{{\footnotesize{$y_{2n+7}$}}}; \draw (4,0) node[below]{{\footnotesize{$y_{2n+9}$}}};
	  \draw (2, 0) arc (-45:-135:1.41cm); \draw (4, 0) arc (45:135:1.41cm); 
	\end{tikzpicture}
\end{minipage}
\end{center}

One can then construct explicit minimal resolutions to compute the Adams $E_2$ page in each case. 
Note that when $n = 1$ mod 3, the resolution can be taken as a degree $2n-1$ shift of a resolution of $\bb{C}P^{\infty}$, which can be learned from Aikawa \cite{Aikawa}.  When $n = 0$ mod 3, the resolution can be taken as a direct sum of a resolution of $\bb{Z}/3$ and that of $\bb{C}P^{\infty}$ followed by a degree shift of $2n+1$. We omit the details 
of constructing resolutions and simply provide the Adams charts.

When $n = 0$ mod 3, the Adams $E_2$-page for $\Sigma\bb{C}P^{\infty}_n$ begins as follows.

\begin{center}
\DeclareSseqGroup\tower {} {
    \class(0,0)
    \foreach \y in {1,...,9} {
        \class(0,\y)
\structline
}}

\begin{sseqdata}[name = 3primary0, Adams grading, yscale = 0.7, xscale = 1.5, y range = {0}{4}, x range = {4}{10}, no x ticks]
\begin{scope}[ background ]
\node at (5,\ymin - 1.5) {2n+1};
\node at (6,\ymin - 1.5) {2n+2};
\node at (7,\ymin - 1.5) {2n+3};
\node at (8,\ymin - 1.5) {2n+4};
\node at (9,\ymin - 1.5) {2n+5};
\end{scope}
\tower(5, 0)
\tower(7, 0)
\tower(9, 0)
\class(8, 1)
\structline(5, 0)(8, 1)
\end{sseqdata}
\printpage[name = 3primary0, scale = 0.95]
\end{center}

When $n = 1$ mod 3, the beginnings of the Adams $E_2$-page for $\Sigma\bb{C}P^{\infty}_n$ is as follows.

\begin{center}
\DeclareSseqGroup\tower {} {
    \class(0,0)
    \foreach \y in {1,...,9} {
        \class(0,\y)
\structline
}}

\begin{sseqdata}[name = 3primary1, Adams grading, yscale = 0.7, xscale = 1.5, y range = {0}{4}, x range = {4}{10}, no x ticks]
\begin{scope}[ background ]
\node at (5,\ymin - 1.5) {2n+1};
\node at (6,\ymin - 1.5) {2n+2};
\node at (7,\ymin - 1.5) {2n+3};
\node at (8,\ymin - 1.5) {2n+4};
\node at (9,\ymin - 1.5) {2n+5};
\end{scope}
\tower(5, 0)
\tower(7, 0)
\tower(9, 1)
\end{sseqdata}
\printpage[name = 3primary1, scale = 0.95]
\end{center}

When $n = 2$ mod 3, the beginnings of the Adams $E_2$-page for $\Sigma\bb{C}P^{\infty}_n$ is as follows.

\begin{center}
\begin{sseqdata}[name = 3primary2, Adams grading, yscale = 0.7, xscale = 1.5, y range = {0}{4}, x range = {4}{10}, no x ticks]
\begin{scope}[ background ]
\node at (5,\ymin - 1.5) {2n+1};
\node at (6,\ymin - 1.5) {2n+2};
\node at (7,\ymin - 1.5) {2n+3};
\node at (8,\ymin - 1.5) {2n+4};
\node at (9,\ymin - 1.5) {2n+5};
\end{scope}
\tower(5, 0)
\tower(7, 0)
\tower(9, 1)
\end{sseqdata}
\printpage[name = 3primary2, scale = 0.95]
\end{center}

In all the cases above, there is no room for any differential $d_r$, for all $r\geq 2$. 
One can therefore read off the desired 3-local stable homotopy groups immediately. Namely,

\begin{itemize}
	\item $\pi^s_{2n+1}(\Sigma\bb{C}P^{\infty}_n) = \bb{Z}$, $\pi^s_{2n+2}(\Sigma\bb{C}P^{\infty}_n) = 0$, and $\pi^s_{2n+3}(\Sigma\bb{C}P^{\infty}_n) = \bb{Z}$.
	\item $\pi^s_{2n+4}(\Sigma\bb{C}P^{\infty}_n)$ exhibits the following 3-fold periodic behavior. It is zero when $n = 0, 1$ mod 3, and is $\bb{Z}/3$ when $n=2$ mod 3.
\end{itemize}

This completes the 3-local computations, and proves Lemma \ref{sthomotopygeneral3}.

\end{document}